\documentclass[11pt]{amsart}
\usepackage{preamble}

\title{Non-vanishing of Poincar\'e Series on Average}
\date{December 28, 2025}

\author{Ned Carmichael}
\address{N. Carmichael: Department of Mathematics, King’s College London, London, WC2R 2LS, UK.}
\email{\href{mailto:ned.carmichael@kcl.ac.uk}
{ned.carmichael@kcl.ac.uk}}

\author{Noam Kimmel}
\address{N. Kimmel: Raymond and Beverly Sackler School of Mathematical Sciences, Tel Aviv University, Tel Aviv 69978, Israel.}
\email{\href{mailto:noamkimmel@mail.tau.ac.il}{noamkimmel@mail.tau.ac.il}}


\begin{document}

\maketitle

\begin{abstract}
We study when Poincar\'e series for congruence subgroups do not vanish identically.
We show that almost all Poincar\'e series with suitable parameters do not vanish when either the weight $k$ or the index $m$ varies in a dyadic interval.
Crucially, analyzing the problem `on average' over these weights or indices allows us to prove non-vanishing in ranges where the index $m$ is significantly larger than $k^2$ - a range in which proving non-vanishing for individual Poincar\'e series remains out of reach of current methods.
\end{abstract}

\setcounter{tocdepth}{1}
{\tableofcontents}

\section{Introduction}

Let $m$, $q$ and $k>2$ be positive integers, and let \(\chi\) be a character modulo \(q\) satisfying the parity condition \(\chi(-1)=(-1)^k\). 
We consider the Poincar\'e series $P_{k,m,q,\chi}(z)$ of weight $k$, index $m$, and character $\chi$ for the Hecke congruence subgroup
$$
\Gamma_0(q) = \SET{
\begin{pmatrix}
    a       & b  \\
    c       & d 
\end{pmatrix}\in \SL \; , \; q\mid c
},
$$
based at the cusp $i\infty$.
This series is defined by
$$
P_{k,m,q,\chi}(z) = 
\sum_{\gamma \in \Gamma_\infty \backslash \Gamma_0(q)}
\overline{\chi}(\gamma) j(\gamma, z)^{-k} e(m\gamma z),
$$
where 
$$
\Gamma_\infty = \left\lbrace \pm\begin{pmatrix}
    1       & *  \\
    0       & 1 
\end{pmatrix} \in \SL \right\rbrace, \quad \chi\left(
\begin{pmatrix}
    a       & b  \\
    c       & d 
\end{pmatrix}
\right)=\chi(d), \quad
j\left(
\begin{pmatrix}
    a       & b  \\
    c       & d 
\end{pmatrix}
,z \right) = cz+d,
$$
and $e(z) = e^{2\pi i z}$.
In the case \(q=1\) we write \(P_{k,m,1}\) for \(P_{k,m,1,\chi}\), since necessarily \(\chi=\chi_0\) is trivial. 
See Iwaniec \cite[\textsection 3]{MR1474964} for background on Poincar\'e series.

A basic yet unresolved question is to determine when these Poincar\'e series vanish identically - a problem that traces back to Poincar\'e's memoir on Fuchsian groups \cite[p.249]{MR1554584}. 
In the classical setting $q=1$ and $k=12$, Lehmer \cite{MR0021027} famously conjectured that the Poincar\'e series \((P_{12,m,1})_{m\geq1}\) never vanish identically. 
This is equivalent to the assertion that the Ramanujan tau function $\tau(n)$ is nonzero for all positive integers $n$. 
Indeed, the space of weight 12 cusp forms for \(\Gamma_0(1) = \SL \) is one-dimensional, spanned by the modular discriminant \(\Delta\). Consequently, by computing the Petersson inner product \(\langle \Delta, P_{12,m,1}\rangle\) one can show that for \(m\geq1\),
$$
P_{12,m,1}=\frac{\Gamma(11)\tau(m)}{(4\pi m)^{11}\langle\Delta,\Delta\rangle} \Delta.
$$

For general $k$ and $q=1$, partial non-vanishing results were obtained by Rankin \cite{MR0597120}. 
He showed that for sufficiently large even weight $k$, the series $P_{k,m,1}$ is not identically zero provided 
$$
m \leq   k^2\exp\left(-B\frac{\log (k)}{\log (\log (k))}\right) ,
$$
for some absolute constant $B > 0$.
Rankin's results have been extended by Lehner \cite{MR0597127} to general Fuchsian groups (albeit with weaker bounds), and Mozzochi \cite{MR0982000} to $\Gamma_0(q)$ with $q>1$.

An improvement to Rankin's and Mozzochi's results, when $q$ is square-free, was given by the second author in \cite{hrj:15573}. 
There, it is shown that $P_{k,m,q,\chi} \not \equiv 0$ for $\chi=\chi_0$ the trivial character and 
$$
m \ll_\epsilon k^2 \frac{q^2}{p_{\max}(q)^{1/2 + \epsilon}},
$$
where $p_{\max}(q)$ is the maximal prime dividing $q$.

However, in the weight aspect, it seems proving non-vanishing for $m>k^2$ is a difficult problem. 
Indeed, both the methods of Rankin and of the second author encounter obstructions at this point. 
It is therefore natural to ask whether the non-vanishing statements remain true beyond $k^2$.
To address this, we consider the problem of non-vanishing `on average': we wish to show that almost all Poincar\'e series are non-vanishing, in some range of $m>k^2$.

In this paper we consider two different averages - over the index $m$ and over the weight $k$.
We show that for suitable parameter choices, almost all Poincar\'e series $P_{k,m,q,\chi}$ do not vanish identically as $m$ or $k$ vary over a dyadic segment, with the other parameters given.
Our results improve upon previous work of Das and Ganguly \cite{MR2997486}, who studied the same problem of non-vanishing `on average'.

In order to show that almost all the Poincar\'e series $P_{k,m,q,\chi}$ do not vanish identically, we study their $L^2$ norms (with respect to the Petersson inner product on \(\Gamma_0(q)\backslash\mathbb{H}\)).
It is known (see \cite[Corollary 3.4]{MR1474964}) that
\begin{multline}   \label{eq:l2_norm}
\|P_{k,m,q,\chi}(z)\|_2^2
\\=
\int_{\Gamma_0(q)\backslash \mathbb{H}}\left|P_{k,m,q,\chi}(z)\right|^2 y^k\frac{\mathrm d x \mathrm d y}{y^2}
=
\frac{\Gamma(k-1)}{(4\pi m)^{k-1}}\left(1+2\pi i^k \Delta_{k,q,\chi}(m,m)\right),
\end{multline}
where
\begin{equation}\label{def:delta}
\Delta_{k,q,\chi}(m,n)\coloneqq \sum_{c\geq1} \frac{S_\chi(m,n;cq)}{cq}J_{k-1}\left(\frac{4\pi \sqrt{mn}}{cq}\right).
\end{equation}
Here $S_\chi(m,n;cq)$ denotes the twisted Kloosterman sum 
\begin{equation}\label{def:kloosterman}
S_\chi(m,n;cq)=\sum_{\substack{a\bmod {cq}\\ (a,cq)=1}}\chi(a)e\left(\frac{ma+n\overline a}{cq}\right),
\end{equation}
where $\overline a$ is the inverse of $a$ mod $cq$, and $J_{k-1}$ denotes the $J$-Bessel function.

Denote 
$$
\widetilde{P}_{k,m,q,\chi} = \left(\frac{\Gamma(k-1)}{(4\pi m)^{k-1}}\right)^{-1/2}P_{k,m,q,\chi}.
$$
Then by \eqref{eq:l2_norm} we have
\begin{equation}\label{eq:l2_norm_normalized}
\|\widetilde{P}_{k,m,q,\chi} \|_2^2 = 1 + 2\pi i^k \Delta_{k,q,\chi}(m,m).
\end{equation}
Our two main theorems give an average asymptotic estimate for the $L^2$ norms of these renormalized Poincar\'e series.

\begin{theorem}\label{thm:k-avg}
Let $m$ and $q$ be positive integers, let $K\geq1$ and suppose $m\leq K^{100}$. 
Let $\delta\in\SET{0,1}$ and let $\chi$ be a character modulo $q$ of conductor $q'\mid q$ satisfying $\chi(-1) = (-1)^\delta$.
Let \(g\) be the multiplicative function defined on prime powers by 
\begin{equation*}
g(p^a)
=
\begin{cases}
1 & \text{ if } a=1, \\
p^{a/4} & \text{ if } a\geq2 \text{ is even}, \\
p^{(a+1)/4} & \text{ if } a\geq3 \text{ is odd}.
\end{cases}
\end{equation*}
Let \(\epsilon>0\) be sufficiently small, and suppose $m$ satisfies 
$$
m\cdot(m,q) \leq  K^{3-200\epsilon}q^{2}/g(q')^2.
$$
Then
\begin{multline*}
\#
\left\{K < k < 2K , \; k\equiv\delta \bmod{2}: \; \left|\|\widetilde{P}_{k,m,q,\chi}(z)\|_2^2 - 1\right|\leq K^{-\epsilon}\right\}
\\= \frac{K}{2}
\left(1- 
O_{\epsilon}\left(K^{-\epsilon}\right)
\right).
\end{multline*}
In particular, 
\begin{equation*}
\#\{K < k < 2K , \; k\equiv\delta \bmod{2}: P_{k,m,q,\chi} \not\equiv 0\}
=
\frac{K}{2}
\left(1-
O_\epsilon\left(K^{-\epsilon}\right)
\right) .
\end{equation*}
\end{theorem}

\begin{remark}
In \cite[Theorem 4.2]{MR2997486}, it was shown that for $m\ll K^{5/2}q$ with $ (m,q) = 1$, a proportion of at least $\frac{1}{2\divis(m)}$ of the Poincar\'e series $P_{k,m,q,\chi_0}$ with $K<k<2K$ and $2\mid k$ do not vanish (where $\sigma_0$ is the divisor function).
\Cref{thm:k-avg} improves this result in several ways.

Firstly, the range of admissible indices $m$ is improved in the $K$ aspect from $K^{5/2}$ to $K^{3 - \epsilon}$.
In the setting of \cite[Theorem 4.2]{MR2997486} (when \(\chi=\chi_0\) is the trivial character and $(m,q)=1$) the $q$ aspect is also improved from $q$ to $q^{2}$.
Secondly, we show that almost all Poincar\'e series do not vanish in the relevant range, as opposed to a proportion of $\frac{1}{2\divis(m)}$, which can tend to 0 for various choices of $m$.
Lastly, our result gives further information on the $L^2$ norms of the Poincar\'e series beyond the non-vanishing.
\end{remark}

\begin{theorem}\label{thm:m-avg}
Let \(k\) and \(q\) be positive integers with \(q\leq k^{100}\), and let \(M\geq1\). 
Let \(\chi\) be a character modulo \(q\) satisfying the parity condition \(\chi(-1)=(-1)^k\). 
Let \(\epsilon>0\) be sufficiently small, and suppose \(M\leq k^{5/2-500\epsilon}q^{2}/f(q)\), where \(f(q)\) denotes the largest square divisor of \(q\). 
Then 
\begin{equation*}
\#\SET{M<m<2M:\left|\|\widetilde{P}_{k,m,q,\chi}(z)\|_2^2-1\right|\leq k^{-\epsilon}}
=M(1-O_{\epsilon}(M^{-\epsilon})).
\end{equation*}
In particular, 
\begin{equation*}
\#\SET{M<m<2M: P_{k,m,q,\chi} \not\equiv 0}
=M(1- O_\epsilon(M^{-\epsilon})).
\end{equation*}
\end{theorem}

\begin{remark}
Previously, in \cite[Theorem 4.1]{MR2997486} non-vanishing had been established for $P_{k,m,q,\chi_0}$ for a positive proportion of arbitrarily large $m$ in short intervals, provided that $k$ and $q$ were both fixed and bounded below by some absolute constant.
\Cref{thm:m-avg} is somewhat different. 
It has the advantage of addressing almost all forms, as opposed to a positive proportion, and can handle large $k$ and potentially large $q$, but restricts the index $m$ in terms of these parameters.
Our result also gives additional information on the \(L^2\) norms.
\end{remark}

The key idea behind the proofs of \Cref{thm:k-avg} and \Cref{thm:m-avg} is to bound the (smoothed) second moment of the squared $L^2$ norms of the Poincar\'e series.
That is, we study 
$(\|\widetilde{P}_{k,m,q,\chi}\|_2^2 - 1)^2$
on average. 
We average over $k\sim K$ for \Cref{thm:k-avg} and over $m\sim M$ for \Cref{thm:m-avg}.
Under the conditions of these theorems, we show these second moments are small, and deduce from this that $\|\widetilde{P}_{k,m,q,\chi}\|_2^2$ is almost always close to $1$.
\subsection*{Acknowledgments}
We thank Stephen Lester and Zeev Rudnick for helpful discussions, comments and corrections.

N.C. was supported by the Additional Funding Programme for Mathematical Sciences, delivered by EPSRC (EP/V521917/1) and the Heilbronn Institute for Mathematical Research. N.K. was supported by the European Research Council (ERC) under the European Union's  Horizon 2020 research and innovation program (Grant agreement No. 786758) and by the ISRAEL SCIENCE FOUNDATION (grant No. 2860/24).

\section{Preliminaries}

\subsection{Notation} 

We use the standard notation \(e(z)=e^{2\pi iz}\).
We denote by \(\overline a \) the inverse of \(a\) modulo \(c\), i.e. \(a\overline a\equiv 1\bmod c\), and we denote \(a^\dagger\equiv a+\overline a \bmod c\). 
We also use the notation \(\mathbf1\{A\}\) to denote the indicator function of \(A\). 
We write \((a,b)\) for the greatest common divisor of two integers \(a\) and \(b\).
In sums (and products), the subscript \(\sum_{d\mid n}\) will always denote a sum (or product) taken over the positive integers \(d\) dividing \(n\).
The subscript \(\sum_{m\leq x}\) will always denote a sum (or product) taken over positive integers \(1\leq m\leq x\).
We denote by $\divis$ the divisor function $\divis(n) = \sum_{d\mid n} 1$.

Throughout, \(g(n)\) denotes the multiplicative function defined in \Cref{thm:k-avg}. 
That is, \(g\) is the multiplicative function taking the following values at prime powers:
\begin{equation*}
g(p^a)
\coloneqq
\begin{cases}
1 & \text{ if } a=1, \\
p^{a/4} & \text{ if } a\geq2 \text{ is even}, \\
p^{(a+1)/4} & \text{ if } a\geq3 \text{ is odd}.
\end{cases}
\end{equation*}
Furthermore, as in \Cref{thm:m-avg}, \(f(n)\) always denotes the largest square divisor of \(n\). 
In other words, \(f\) is the multiplicative function defined at prime powers by
\begin{equation*}
f(p^a)\coloneqq \begin{cases} p^a &\text{ if } a \text{ is even},\\ p^{a-1} &\text{ if } a \text{ is odd}.\end{cases}
\end{equation*}

We will use the strict $O$ notations, so that $F(x) = \BigO{G(x)}$ if one has $|F(x)| \leq C |G(x)|$ for some constant $C>0$ and for all valid inputs $x$.
We will also write $F(x) \ll G(x)$ to indicate $F(x) = \BigO{G(x)}$.
If the implied constants depend on some parameters $p_1,p_2,...$, this will either be explicitly stated, or indicated by subscripts such as $O_{p_1,_2,...}$ or $\ll_{p_1,p_2,...}$.

\subsection{Smooth functions}\label{sec:smooth-func}
Throughout, we let \(u\) be a smooth test function satisfying the following:
\begin{equation}\label{def:smooth-u}
\begin{gathered}
\mathrm{supp}(u) 
\subset [1/2, 3],\\
0 \leq u(x) \leq 1, \\
u(x) = 1 \; \mathrm{for} \; x\in [1,2].
\end{gathered}
\end{equation}

We will denote by $\widehat{u}$ the Fourier transform of $u$:
$$
\widehat{u}(t) = \int_\RR u(x) e(-xt)\mathrm dx.
$$
By properties of the Fourier transform it is known that
$$
u^{(j)}(x) = (2\pi i)^j \int_\RR \widehat{u}(t) t^j e(xt)\mathrm dt.
$$
In particular, taking $x=0$ implies that for all $j\geq 0$:
\begin{equation}\label{eq:fourier-derivative}
\int_\RR \widehat{u}(t) t^j \mathrm dt  
=
(2\pi i)^{-j}u^{(j)}(0) = 0.
\end{equation}

Since $u$ is fixed, we have that
$$
u^{(j)} \ll_j 1 \text{ for all } j\geq 0.
$$
By partial integration this implies
\begin{equation}\label{eq:u-decay}
\widehat{u}(t) \ll_j (1 + |t|)^{-j} \text{ for all } j\geq 0.
\end{equation}
This bound and a change of variable yields
\begin{equation}\label{eq:fourier-bound}
K\int_\RR    \left|\widehat{u}(Kt) t^j \right| \mathrm dt
\ll_j K^{-j}
\end{equation}
for all $j\geq 0$ and $K>0$.

\subsection{Bessel functions}
We will require various bounds and properties of the \(J\)-Bessel function, which we collect in this section.
The Bessel functions appear throughout (due to their presence in the expression for $\left\|P_{k,m,q,\chi}\right\|_2^2$), and are defined (see \cite[p.40]{watson42}) for \(\nu\geq0\) and \(z\geq0\) by
\begin{equation}\label{def:J-bessel-Taylor}
J_\nu(z)=\sum_{\ell\geq 0}\frac{(-1)^\ell}{\ell!\Gamma(\nu+1+\ell)}\Big(\frac{z}{2}\Big)^{\nu+2\ell}.
\end{equation}
For all integers \(n\) and all real \(z\), we also have the integral representation \cite[p.20]{watson42}:
\begin{equation}\label{eq:bessel-intrep}
J_n(z)=\frac{1}{\pi}\int_0^\pi \cos(n\theta-z\sin\theta)\mathrm d\theta, \text{ which implies } |J_n(z)|\leq 1 \text{ for all real } z.
\end{equation}

\begin{lemma}[Bessel Function Bounds]\label{lem:bessel-bounds}

Let \(\nu\geq 2\). 
We have the following bounds.

\begin{enumerate}[label=(\roman*)] 

\item \label{besi} If \(0\leq z\leq (\nu+1)/4\), then
\begin{equation*}
J_\nu(z)\ll z^2\exp\Big(-\frac{14\nu}{13}\Big).
\end{equation*}

\item \label{besii} If \(0\leq z\leq (\nu+1)-(\nu+1)^{1/3+\delta}\) for some \(0<\delta\leq 2/3\), then
\begin{equation*}
J_\nu(z)\ll \exp(-\nu^\delta).
\end{equation*}

\item \label{besiii} Uniformly for \(z\geq0\), we have
\begin{equation*}
J_\nu(z)\ll \nu^{-1/3}.
\end{equation*}

\item \label{besiv} If \(z\geq \nu+\nu^{1/3+\delta}\) for some \(\delta>0\), then
\begin{equation*}
J_\nu(z)\ll_\delta (z^2-\nu^2)^{-1/4}.
\end{equation*}
\end{enumerate}

\end{lemma}

\begin{proof}

Parts \ref{besi}, \ref{besii} and \ref{besiii} are contained in \cite[Lemma A.2]{paper1}. 
Part \ref{besiv} follows immediately from the asymptotics given in \cite[Lemma A.1]{paper1}. 
Note that whilst the results in \cite{paper1} are formulated for \(\nu\to\infty\), they also hold for \(\nu\geq2\) (if \(\nu\) is bounded then \(J_\nu\) is also bounded).
\end{proof}

\begin{lemma}[Neumann's addition theorem]\label{lem:Neumann-addition-theorem-original-version}
Let \(x_1, x_2\) and \(\theta\) be real numbers. 
Then
$$
\sum_{n\in\ZZ}J_n(x_1)J_n(x_2)e^{-in\theta}
=
J_{0}\left(\sqrt{x_1^2 + x_2^2 -2x_1x_2\cos(\theta)}\right).
$$
\end{lemma}
\begin{proof}
See \cite[\textsection 11.2]{watson42}.
\end{proof}

As a consequence we have:
\begin{lemma}\label{lem:Neumann-addition-theorem}
Let $K>0$, $\delta\in\SET{0,1}$ and $u$ be as in \Cref{sec:smooth-func}.
Then for any real numbers \(x_1\) and \(x_2\) we have
$$
\sum_{n\equiv\delta\bmod{2}}
u\left(\frac{n}{K}\right)
J_n(x_1)J_n(x_2)
=
\frac{1}{2}\left(
I_1 + (-1)^\delta I_2
\right)
$$
where
$$
I_1 
= K\int_\RR\widehat{u}(Kt) J_0\left(\sqrt{x_1^2 + x_2^2-2x_1x_2\cos(2\pi t)}\right)\mathrm dt
$$
and
$$
I_2 
= K\int_\RR\widehat{u}(Kt) J_0\left(\sqrt{x_1^2 + x_2^2 + 2x_1x_2\cos(2\pi t)}\right)\mathrm dt.
$$
\end{lemma}
\begin{proof}
From the inverse Fourier transform we have that
$$
u\left(\frac{n}{K}\right) 
=
K \int_{\RR}\widehat{u}(Kt)e(nt)\mathrm d t.
$$
Thus we have that
\begin{equation}\label{eq:u-j-sum}
\sum_{n\in\ZZ}
u\left(\frac{n}{K}\right)
J_n(x_1)J_n(x_2)
=
K
\sum_{n\in\ZZ}
\int_{\RR}\widehat{u}(Kt)e(nt)\mathrm d t
\cdot J_n(x_1)J_n(x_2).
\end{equation}
From the fast decay of $\widehat{u}(t)$ as $|t|\rightarrow\infty$ (see \eqref{eq:u-decay}), and the fast decay of $J_n(x)$ as $|n|\rightarrow\infty$ (see \Cref{lem:bessel-bounds} part \ref{besii}), the right hand side of \eqref{eq:u-j-sum} converges absolutely.
Thus, changing the order of summation and integration, we get
$$
\sum_{n\in\ZZ}
u\left(\frac{n}{K}\right)
J_n(x_1)J_n(x_2)
=
K
\int_{\RR}\widehat{u}(Kt)
\left(\sum_{n\in\ZZ}e(nt) J_n(x_1)J_n(x_2)\right) \mathrm d t.
$$
Applying \Cref{lem:Neumann-addition-theorem-original-version} we get that
$$
\sum_{n\in\ZZ}
u\left(\frac{n}{K}\right)
J_n(x_1)J_n(x_2)
=
K
\int_{\RR}\widehat{u}(Kt)
J_0\left(\sqrt{x_1^2 + x_2^2-2x_1x_2\cos(2\pi t)}\right)\mathrm d t,
$$
which is precisely the definition of $I_1$. 

The same argument shows that
\begin{align*}
\sum_{n\in\ZZ} 
u\left(\frac{n}{K}\right)&(-1)^n
J_n(x_1)J_n(x_2)
\\&=
K
\int_{\RR}\widehat{u}(Kt)
J_0\left(\sqrt{x_1^2 + x_2^2-2x_1x_2\cos(2\pi t + \pi)}\right)\mathrm d t
\\&=
K
\int_{\RR}\widehat{u}(Kt)
J_0\left(\sqrt{x_1^2 + x_2^2 + 2x_1x_2\cos(2\pi t)}\right)\mathrm d t
= I_2.
\end{align*}
The result then follows by adding/subtracting these expressions for $I_1$ and $I_2$.
\end{proof}

\begin{lemma}\label{lem:bessel0_adition_theorem}
Let $A,B\in \RR$ with $A>0$ and $A+B \geq 0$, and let $N\in \NN$.
Then 
$$
J_0\left(2\sqrt{A+B}\right)
=
\sum_{0 \leq n < N}
\frac{(-1)^n}{n!}\left(\frac{B}{\sqrt{A}}\right)^n J_n\left(2\sqrt{A}\right)
+O\left(
\left(\frac{|B|}{\sqrt{A}}\right)^N
\right).
$$

\end{lemma}

\begin{proof}
From the Taylor expansion of the \(J\)-Bessel functions \eqref{def:J-bessel-Taylor}, we have that
\begin{multline*}  
J_0\left(2\sqrt{A+B}\right)
=
\sum_{\ell\geq 0} \frac{(-1)^\ell}{(\ell!)^2}(A+B)^\ell
=
\sum_{\ell\geq 0} \frac{(-1)^\ell}{(\ell!)^2}
\sum_{0\leq n\leq \ell}
\binom{\ell}{n}B^n A^{\ell-n}
\\=
\sum_{n\geq0}\frac{1}{n!}B^n
\sum_{\ell\geq n}\frac{(-1)^\ell}{\ell!(\ell-n)!} A^{\ell-n}.
\end{multline*}
Replacing the variable \(\ell\) by \(\ell+n\), we now obtain
\begin{equation*}
J_0\left(2\sqrt{A+B}\right)
=
\sum_{n\geq 0}
\frac{(-1)^n}{n!}B^{n}
\sum_{\ell\geq 0}\frac{(-1)^\ell}{\ell!(\ell+n)!}A^\ell
=
\sum_{n\geq 0}
\frac{(-1)^n}{n!}\left(\frac{B}{\sqrt{A}}\right)^n J_n\left(2\sqrt{A}\right).
\end{equation*}

Truncating the last sum at $n < N$ and using the bound $|J_n(x)| \leq 1$ (see \eqref{eq:bessel-intrep}) gives the result in the case where $B/\sqrt{A} \leq 1$ since 
\begin{multline*}
\left|\sum_{n\geq N}
\frac{(-1)^n}{n!}\left(\frac{B}{\sqrt{A}}\right)^n J_n\left(2\sqrt{A}\right)\right|
\\\leq 
\sum_{n \geq N}\frac{1}{n !} \left(\frac{|B|}{\sqrt{A}}\right)^n
\leq \left(\frac{|B|}{\sqrt{A}}\right)^N \sum_{n \geq N}\frac{1}{n !} \ll \left(\frac{|B|}{\sqrt{A}}\right)^N.
\end{multline*}

The statement also holds in the case $B/\sqrt{A} > 1$ since in this case
\begin{multline*}
\left|J_0\left(2\sqrt{A+B}\right)
-
\sum_{0 \leq n < N}
\frac{1}{n!}\left(\frac{|B|}{\sqrt{A}}\right)^n J_n\left(2\sqrt{A}\right)\right|
\\ 
\ll
1 + \left(\frac{|B|}{\sqrt{A}}\right)^{N-1} 
\sum_{0\leq n<N}\frac{1}{n!}
\ll \left(\frac{|B|}{\sqrt{A}}\right)^{N},
\end{multline*}
by the triangle inequality and the bound \(|J_n(x)|\leq 1\). 
\end{proof}

\subsection{Kloosterman sums}
We will also require bounds for the twisted Kloosterman sums \eqref{def:kloosterman}
which appear in the expression for $\left\|P_{k,m,q,\chi}\right\|_2^2$.
The bounds we give are slightly more precise statements of the bounds given by Knightly and Li \cite[\textsection 9]{MR3099744}.

\begin{lemma}[Weil bound for twisted Kloosterman sums]\label{lem:twisted_kloosterman_bound}
Let $m$, $n$, $c$ and $q$ be positive integers. 
Let $\chi$ be a character modulo $q$, of conductor $q'\mid q$.
Then
\begin{equation*}
\left| S_\chi(n,m;cq)\right|
\leq 
4 \divis(cq) (m,n,cq)^{1/2}(cq)^{1/2} g(q').
\end{equation*}
\end{lemma}

This improves \cite[Theorem 9.2]{MR3099744}, in which the factor \(g(q')\) is replaced by the factor \((q')^{1/4}\prod_{p\mid q'} p^{1/4}\).
Clearly \(g(q')\leq (q')^{1/4}\prod_{p\mid q'} p^{1/4}\), with equality if and only if \(q'\) has prime factorization \(q'=p^{a_1}\cdots p^{a_r}\) for \(a_i\geq3\) odd integers.
Notably, if the conductor \(q'\) is square-free then \(g(q')=1\).
\Cref{lem:twisted_kloosterman_bound} will be deduced from the results in \cite[\textsection 9]{MR3099744}. 
The argument is almost entirely the same - one simply requires the following sharper estimate for the Kloosterman sums of prime power modulus. 

\begin{lemma}\label{lem:twisted-kloosterman-bound-primes}
Let \(p^\ell \) be a prime power, and let \(m\) and \(n\) be positive integers such that \(p\nmid n\) and \(p^\ell \nmid m\). 
Let \(\chi\) be a character modulo \(p^\ell \) of conductor \(p^\gamma\mid p^\ell \).
Then 
\begin{equation*}
|S_\chi(n,m;p^\ell )|\leq 2(2,p)^2 p^{\ell /2}g(p^\gamma).
\end{equation*}
\end{lemma}

\begin{proof}
If \(\ell =1\), then the lemma is an extension of the Weil bound due to Chowla, cf. \cite[Proposition 9.4]{MR3099744}.

It remains to prove the lemma in the case \(\ell \geq 2\).
First suppose \(p\) is odd and \(\ell \geq2\) is even.
By \cite[Proposition 9.7, (3)]{MR3099744}, if \(\ell \geq 2\) is even and \(\gamma\leq \ell -1\) then \(|S_\chi(n,m;p^\ell )|\leq 2p^{\ell /2}\). 
If \(\ell =\gamma\geq 2\) is even then by \cite[Proposition 9.7, (1)]{MR3099744}, \(|S_\chi(n,m;p^\ell )|\leq 2p^{3\ell /4}=2p^{\ell /2}g(p^\gamma)\). 
Next suppose \(p\) is odd and \(\ell \geq 3\) is odd. 
By \cite[Proposition 9.8, (3)]{MR3099744}, if \(\ell \geq 3\) is odd and \(\gamma\leq \ell -1\) then \(|S_\chi(n,m;p^\ell )|\leq 2p^{\ell /2}\). 
If \(\ell =\gamma\geq 3\) is odd then by \cite[Proposition 9.8, (1)]{MR3099744}, \(|S_\chi(n,m;p^\ell )|\leq 2p^{(3\ell +1)/4}=2p^{\ell /2}g(p^\gamma)\). 
So for odd \(p\) and any \(0\leq \gamma\leq \ell \) one has \(|S_\chi(n,m;p^\ell )|\leq 2p^{\ell /2}g(p^\gamma)\).

Suppose now \(p=2\).
If \(\gamma\leq \ell -2\) then \cite[Proposition 9.7, (4) \& Proposition 9.8, (4)]{MR3099744} give \(|S_\chi(n,m;2^\ell )|\leq 8 \cdot 2^{\ell /2}\). 
If \(\ell \) is even and \(\ell -1\leq \gamma\leq \ell \), \cite[Proposition 9.7, (2)]{MR3099744} shows \(|S_\chi(n,m;p^\ell )|\leq 4 \cdot 2^{3\ell /4}\leq 4\cdot 2^{\ell /2} 2^{(\gamma+1)/4}\leq 8 \cdot 2^{\ell /2}g(2^\gamma)\).
If \(\ell \) is odd and \(\ell -1\leq \gamma\leq \ell \), \cite[Proposition 9.8, (2)]{MR3099744} shows \(|S_\chi(n,m;p^\ell )|\leq 4\cdot 2^{(3\ell +1)/4}\leq 4\cdot 2^{\ell /2}2^{(\gamma+2)/4}\leq 8 \cdot 2^{\ell /2}g(2^\gamma)\). 
So for any \(0\leq \gamma\leq \ell \) one has \(|S_\chi(n,m;2^\ell )|\leq 8 \cdot 2^{\ell /2}g(p^\gamma)\).
\end{proof}

The deduction of \Cref{lem:twisted_kloosterman_bound} from \Cref{lem:twisted-kloosterman-bound-primes} is exactly as in \cite[\textsection 9]{MR3099744}. 
We include some brief details of this.

\begin{proof}[Proof of \Cref{lem:twisted_kloosterman_bound}]
By the twisted multiplicativity of the twisted Kloosterman sums \cite[Corollary 9.14]{MR3099744}, which is a consequence of the Chinese Remainder Theorem, it suffices to handle the case of prime power modulus. 
More precisely, we must show that for any positive integers \(n\) and \(m\), a prime power \(p^\ell \) and a character \(\chi\) of conductor \(p^\gamma \mid p^\ell \),
\begin{equation}\label{eq:weil-bound-on-primes} 
|S_\chi(n,m;p^\ell )|\leq 2(2,p)^2(n,m,p^\ell )^{1/2}p^{\ell /2}g(p^\gamma).
\end{equation}

In the case \(p^\ell \mid n\) and \(p^\ell  \mid m\), the result is trivial (since \((n,m,p^\ell )=p^\ell \)). 
Otherwise, write \(p^\alpha=(n,m,p^\ell )\) with \(0\leq \alpha<\ell \).
Put \(n=p^\alpha n'\) and \(m=p^\alpha m'\). 
Since \(S_\chi(n,m;p^\ell )=S_{\overline \chi}(m,n;p^\ell )\) we may assume without loss of generality that \(p\nmid n'\). 
Consequently
\begin{multline}\label{eq:kloosterman-gcd-reducing}
S_\chi(n,m;p^\ell )
\\=
\sum_{\substack{a\bmod p^\ell \\p\nmid a}} 
\chi(a) e\left(\frac{n'a+m'\overline a}{p^{\ell -\alpha}}\right)
=
\sum_{\substack{a\bmod p^{\ell -\alpha}\\p\nmid a}} e\left(\frac{n'a+m'\overline a}{p^{\ell -\alpha}}\right)
\sum_{h\bmod p^\alpha} \chi(a+hp^{\ell -\alpha}).
\end{multline}

If \(p^\gamma> p^{\ell -\alpha}\), the inner sum in \eqref{eq:kloosterman-gcd-reducing} vanishes (see \cite[Theorem 9.4]{montgomeryvaughan}) and we are done.
If alternatively \(p^\gamma\mid p^{\ell -\alpha}\), then the inner sum in \eqref{eq:kloosterman-gcd-reducing} is equal to \(p^\alpha\chi(a)\) and we obtain
\begin{equation}\label{eq:kloosterman-gcd-factoring}
S_\chi(n,m;p^\ell )=p^\alpha S_\chi(n',m';p^{\ell -\alpha}).
\end{equation}
Recall \(p\nmid n'\). 
If \(p^{\ell -\alpha}\mid m'\) then
\[S_\chi(n',m';p^{\ell-\alpha})=\overline{\chi}(n')\sum_{\substack{a\bmod p^{\ell-\alpha}\\ p\nmid a}}\chi(a)e\left(\frac{a}{p^{\ell-\alpha}}\right).\]
This is either a Gauss sum or a geometric series, accordingly as \(\gamma\geq 1\) or \(\gamma=0\). In any case, one has \(|S_\chi(n',m';p^{\ell-\alpha})|\leq p^{(\ell-\alpha)/2}\), cf. \cite[Propositions 9.4 \& 9.10]{MR3099744}, whence \eqref{eq:kloosterman-gcd-factoring} implies \(|S_\chi(n,m;p^\ell)|\leq p^\alpha |S_\chi(n',m';p^{\ell-\alpha})|\leq p^{(\ell+\alpha)/2}\), as claimed.
If on the other hand \(p^{\ell -\alpha}\nmid m'\), we apply \Cref{lem:twisted-kloosterman-bound-primes}, which gives \(|S_\chi(n',m';p^{\ell -\alpha})|\leq 2(2,p)^2 p^{(\ell -\alpha)/2}g(p^\gamma)\). 
The lemma now follows from this bound and \eqref{eq:kloosterman-gcd-factoring}.
\end{proof}
\section{Reduction to Second Moment Estimates}\label{sec:deductions}

In this section, we state the two key technical results required in the proofs of \Cref{thm:k-avg} and \Cref{thm:m-avg}, and show that these imply the Theorems. 
In both cases, we prove second moment bounds for the quantity
$$
\left(\|\widetilde{P}_{k,m,q,\chi}\|_2^2 - 1\right)^2=(-1)^k 4\pi^2\Delta_{k,q,\chi}(m,m)^2.
$$
(This equality follows from \eqref{eq:l2_norm_normalized}.)
For \Cref{thm:k-avg}, we average over the weight, and for \Cref{thm:m-avg}, we average over the index. 
To this end, with \(u\) being the smooth and compactly supported test function given in (\ref{def:smooth-u}), for parameters \(K,M\geq1\) we define the second moments
\begin{multline}\label{eq:sigmaK-def}
\sigma_{K}^2
\coloneqq \frac{2}{K}
\sum_{k\equiv \delta \bmod 2}u\left(\frac{k-1}{K}\right) \left(\|\widetilde{P}_{k,m,q,\chi}\|_2^2 - 1\right)^2\\
=(-1)^\delta \frac{8\pi^2}{K}\sum_{k\equiv \delta \bmod 2}u\left(\frac{k-1}{K}\right)\Delta_{k,q,\chi}(m,m)^2,
\end{multline}
and 
\begin{equation}\label{eq:sigmaM-def}
\sigma_M^2\coloneqq \frac 1M \sum_{m} u\left(\frac{m}{M}\right)\left(\|\widetilde P_{k,m,q,\chi}\|_2^2-1\right)^2=(-1)^k \frac {4\pi^2}M \sum_m u\left(\frac mM\right)\Delta_{k,q,\chi}(m,m)^2,
\end{equation}
where $\delta\in\SET{0,1}$ is the parity of $\chi$ (i.e. $\chi(-1) = (-1)^\delta$).

In \Cref{sec:weight-ave} we will prove the following second moment bound for \(\sigma_{K}^2\). 
\begin{proposition}\label{prop:seond-moment-weight}
Let $m$ and $q$ be positive integers, let $K\geq1$ and suppose $m\leq K^{100}$. 
Let $\delta\in\SET{0,1}$ and let $\chi$ be a character modulo $q$ of conductor $q'\mid q$ satisfying $\chi(-1) = (-1)^\delta$.
Then for any $\epsilon > 0$, one has
$$
\sigma_{K}^2
\ll_\epsilon
\frac{m^{1+\epsilon}(m,q)g(q')^2}{K^3 q^2}
+ K^{-1000}.
$$
\end{proposition}

\begin{proof}[Proof of \Cref{thm:k-avg}, assuming \Cref{prop:seond-moment-weight}]
Fix \(\epsilon>0\). One has 
\begin{multline}\label{eq:setKsizes}
\#\left\{K<k<2K, k\equiv\delta \bmod{2} : \left|\|\widetilde{P}_{k,m,q,\chi}(z)\|_2^2 - 1\right|\leq K^{-\epsilon}\right\}
=\frac K2 +O(1)
\\-
\#\left\{K+1<k<2K+1, k\equiv \delta\bmod{2} : \left|\|\widetilde{P}_{k,m,q,\chi}(z)\|_2^2 - 1\right|> K^{-\epsilon}\right\}.
\end{multline}
Since \(u\) is non-negative and \(u((k-1)/K)=1\) for \(K+1<k<2K+1\), we bound
\begin{multline}\label{eq:setKbound}
\#\left\{K+1<k<2K+1, k\equiv \delta\bmod 2:\left|\|\widetilde{P}_{k,m,q,\chi}(z)\|_2^2-1\right|> K^{-\epsilon}\right\}\\
\leq \sum_{k} 
u\left(\frac{k-1}{K}\right) 
\left|\|\widetilde{P}_{k,m,q,\chi}(z)\|_2^2-1\right|^2 
K^{2\epsilon}
\ll K^{1+2\epsilon}
\sigma_{K}^2.
\end{multline}
Since we assume $m(m,q) \leq  K^{3-200\epsilon}q^{2}/g(q')^2$ and $m\leq K^{100}$, \Cref{prop:seond-moment-weight} now gives
\begin{equation*}
K^{1 + 2\epsilon}\sigma_{K}^2 \ll_\epsilon
K^{1+2\epsilon}\left( \frac{m^{1+\epsilon}(m,q)g(q')^2}{K^3 q^2} + K^{-1000}\right)
\ll_\epsilon K\left(K^{-98\epsilon} + K^{-1000 + 2\epsilon}\right).
\end{equation*}
Combining this with \eqref{eq:setKbound} and \eqref{eq:setKsizes}, the result follows. 
\end{proof}

Next, in \Cref{sec:index_ave}, we prove the following second moment bound for \(\sigma_M^2\).

\begin{proposition}\label{prop:second-moment-index}
Let \(k\) and \(q\) be positive integers, let \(M\geq1\) and suppose \(M\leq k^{1000}\). Let \(\chi\) be a character modulo \(q\) satisfying \(\chi(-1)=(-1)^k\). 
Then for any \(\epsilon>0\), one has
\begin{equation*}
\sigma_M^2\ll_\epsilon \frac{M^{1/2+\epsilon}f(q)^{1/2}}{k^{3/2}q}+\frac{M^{1+\epsilon}f(q)}{k^{5/2}q^2}+k^{-1000}.
\end{equation*}
\end{proposition}

\begin{proof}[Proof of \Cref{thm:m-avg}, assuming \Cref{prop:second-moment-index}]
Fix \(\epsilon>0\). One has
\begin{multline}\label{eq:setMsizes}
\#\left\{M<m<2M:\left|\|\widetilde{P}_{k,m,q,\chi}(z)_2^2-1\|\right|\leq k^{-\epsilon}\right\}\\
=M+O(1)-\#\left\{M<m<2M:\left|\|\widetilde{P}_{k,m,q,\chi}(z)\|_2^2-1\right|> k^{-\epsilon}\right\}.
\end{multline}
Since \(u\) is non-negative and \(u(m/M)=1\) for \(M<m<2M\), we bound
\begin{multline}\label{eq:setMbound}
\#\left\{M<m<2M:\left|\|\widetilde{P}_{k,m,q,\chi}(z)\|_2^2-1\right|> k^{-\epsilon}\right\}\\
\leq 
\sum_{m} 
u\left(\frac{m}{M}\right) 
\left|\|\widetilde{P}_{k,m,q,\chi}(z)\|_2^2-1\right|^2  k^{2\epsilon}
\ll M 
\sigma_M^2 k^{2\epsilon}.
\end{multline}
Since we assume \(M\leq k^{5/2-500\epsilon}q^{2}/f(q)\) and \(q\leq k^{100}\), it follows \(M\leq k^{5/2}q^2\leq k^{405/2}\). 
\Cref{prop:second-moment-index} is therefore applicable, and gives 
\begin{equation}\label{eq:sigmaMbound-final}
M \sigma_M^2 k^{2\epsilon}
\ll_\epsilon 
\left(
k^{-1/4-250\epsilon}+k^{-500\epsilon}
\right)
M^{1+\epsilon} k^{2\epsilon}+Mk^{-1000+2\epsilon}.
\end{equation}
Note \(M\leq k^{405/2}\) implies \(k^{-498\epsilon}\leq k^{-405\epsilon}\leq M^{-2\epsilon}\) and \(k^{-1000+2\epsilon}\leq k^{-810}\leq M^{-4}\) (for sufficiently small \(\epsilon\)).
So \eqref{eq:sigmaMbound-final} shows \(M\sigma_M^2 k^{2\epsilon}\ll_\epsilon M^{1-\epsilon}\).
Combining this with (\ref{eq:setMbound}) and replacing in (\ref{eq:setMsizes}), the result follows.
\end{proof}
\section{Weight Average: Proof of \texorpdfstring{\Cref{thm:k-avg}}{Theorem 1.1}}\label{sec:weight-ave}

In this section we prove \Cref{prop:seond-moment-weight}, which, as shown in \cref{sec:deductions}, implies \Cref{thm:k-avg}.
Throughout this section, we let \(m\) and \(q\) be positive integers, let \(K\geq1\), and suppose \(m\leq K^{100}\). 
Let $\chi$ be a character modulo $q$ of conductor $q'\mid q$ and parity $\delta\in\SET{0,1}$, i.e. $\chi(-1) = (-1)^\delta$.
The second moment \(\sigma_K^2\) (defined in \eqref{eq:sigmaK-def}) depends upon these \(m,q, K\) and \(\chi\). 
All subsequent lemmas in this section are uniform in these parameters.
Recall from the definition \eqref{eq:sigmaK-def} and the definition of \(\Delta_{k,q,\chi}\) given in \eqref{def:delta} that
\begin{align} \label{eq:sigmaK2-non-truncated}
\sigma_K^2 
&=(-1)^\delta \frac {8\pi^2}{K}
\sum_{k\equiv \delta \bmod 2} 
u\left(\frac{k-1}{K}\right)\Delta_{k,q,\chi}(m,m)^2
\\ \nonumber
&=
(-1)^\delta
\frac{8\pi^2}{K}
\sum_{k\equiv\delta\bmod 2}u\left(\frac{k-1}{K}\right)
\\ \nonumber
& \quad \times
\sum_{c_1,c_2\geq 1}
\frac{S_\chi(m,m,qc_1)}{qc_1}
\frac{S_\chi(m,m,qc_2)}{qc_2}
J_{k-1}\left(\frac{4\pi m}{qc_1}\right)
J_{k-1}\left(\frac{4\pi m}{qc_2}\right). 
\end{align}

Our strategy to bound \(\sigma_K^2\) is to interchange summation and apply Neumann's addition theorem (\Cref{lem:Neumann-addition-theorem}) to transform the sums of Bessel functions. 
This gives an expression for \(\sigma_K^2\) in terms of Kloosterman sums and oscillatory integrals. 
Applying the Weil bound for the Kloosterman sums and trivially bounding the oscillatory integrals now recovers the bound \(\sigma_K^2\ll_\epsilon m(m,q)g(q')^2 (mKq)^\epsilon/(Kq)^2\), which is essentially the same bound one obtains by applying the Weil bound for the Kloosterman sums and sharp bounds for the Bessel functions for each $k\sim K$ individually. 
However, with more careful analysis of the oscillatory integrals arising, we capture additional cancellation that leads to the improved bound of \Cref{prop:seond-moment-weight}.

The first step is to transform the sums \eqref{eq:sigmaK2-non-truncated}. 
This is done in the following lemma. 

\begin{lemma}\label{lem:sigmaK-transformed}
One has 
\begin{multline*}
\sigma_K^2 
=
(-1)^\delta 8\pi^2
\sum_{c_1,c_2 \leq \frac{200m}{Kq}}
\frac{S_\chi(m,m,qc_1)}{qc_1}
\frac{S_\chi(m,m,qc_2)}{qc_2}
\\ 
\times
\left( \mathcal I_{m,K,q}^-(c_1,c_2) - (-1)^\delta \mathcal I_{m,K,q}^+(c_1,c_2)\right)
+ O\left(e^{-K/2}\right), 
\end{multline*}
where 
\begin{equation*}
\mathcal I_{m,K,q}^\pm(c_1,c_2)
=
\int_\RR\widehat{u}(Kt)J_0\left(4\pi m\sqrt{\frac{1}{q^2c_1^2} + \frac{1}{q^2c_2^2} \pm 2\frac{\cos(2\pi t)}{q^2c_1c_2}}\right)\mathrm{d}t.
\end{equation*}
\end{lemma}

\begin{proof}
We begin by truncating the sum over $c_1,c_2$ in \eqref{eq:sigmaK2-non-truncated}.
Note that if $4\pi m / (qc) \leq k/4$ then part \ref{besi} of \Cref{lem:bessel-bounds} gives 
$$
J_{k-1}\left(\frac{4\pi m}{qc}\right)
\ll 
\left(\frac{m}{qc}\right)^2 e^{-14k/13}.
$$
Consequently, since we assume $m\leq K^{100}$, and since $k\geq K/2$ from the support of $u$, the contribution of $c_1,c_2 \geq 200 m/(Kq)$ to \eqref{eq:sigmaK2-non-truncated} is seen to be $O(e^{-K/2})$, where we used the trivial bounds for Kloosterman sums and the bound \(|J_{k-1}(4\pi m/(qc))|\leq 1\) for any \(c\leq 200m/(Kq)\) (see \eqref{eq:bessel-intrep}). 
Truncating the sums in \eqref{eq:sigmaK2-non-truncated} and interchanging the order of summation, we thus obtain
\begin{multline}\label{eq:sigmaK2-truncated} 
\sigma_K^2 
=(-1)^\delta
\frac{8\pi^2}{K}
\sum_{c_1,c_2\leq \frac{200m}{Kq}}
\frac{S_\chi(m,m,qc_1)}{qc_1}
\frac{S_\chi(m,m,qc_2)}{qc_2}
\\ \times
\sum_{k\equiv\delta\bmod 2}u\left(\frac{k-1}{K}\right)
J_{k-1}\left(\frac{4\pi m}{qc_1}\right)
J_{k-1}\left(\frac{4\pi m}{qc_2}\right)
+ O\left(e^{-K/2}\right).
\end{multline}
The lemma now follows at once upon applying \Cref{lem:Neumann-addition-theorem} to transform the inner sum over $k$ in \eqref{eq:sigmaK2-truncated}.
\end{proof}

Our next goal is to bound the integrals $\mathcal I_{m,K,q}^\pm(c_1,c_2)$ appearing in \Cref{lem:sigmaK-transformed}, which is done in the following lemma.

\begin{lemma}\label{lem:Ibounds}
For any positive integers \(c_1\) and \(c_2\), we have the following bounds for \(\mathcal I_{m,K,q}^\pm (c_1,c_2)\).
Firstly, 
\begin{equation}\label{eq:I-trivial-bound}
\mathcal I_{m,K,q}^\pm(c_1,c_2)
\ll K^{-1}.
\end{equation}
Secondly, provided \(c_1\neq c_2\), for any positive integers \(N\) and \(N'\) we have
\begin{equation}\label{eq:I_bound}
\mathcal I_{m,K,q}^\pm(c_1,c_2) \ll_{N,N'} \left(\frac{m}{q|c_1 \pm c_2| K^2}\right)^{N'} + 
K^{-N}.
\end{equation}
\end{lemma}

\begin{proof}
The bound \eqref{eq:I-trivial-bound} follows immediately from the bound $|J_0(x)|\leq 1$ (see \eqref{eq:bessel-intrep}) and \eqref{eq:fourier-bound}:
\begin{equation*}
|\mathcal I_{m,K,q}^\pm(c_1,c_2)|
\leq 
\int_\RR\left|\widehat{u}(Kt)\right|\mathrm{d}t
\ll K^{-1}.
\end{equation*}

The second bound \eqref{eq:I_bound} is obtained by approximating the $J_0$ term by a polynomial in $t$ near $t=0$.

We begin by truncating the integral
$$
\mathcal I_{m,K,q}^\pm(c_1,c_2)
=
\int_\RR\widehat{u}(Kt)J_0\left(4\pi m\sqrt{\frac{1}{q^2c_1^2} + \frac{1}{q^2c_2^2} \pm 2\frac{\cos(2\pi t)}{q^2c_1c_2}}\right)\mathrm{d}t
$$
to the range $|t|\leq \frac{1}{100}$.
Using \eqref{eq:u-decay} and the bound $|J_0(x)|\leq 1$ we have that for any positive integer \(N\),
\begin{multline*} 
\left|\int_{|t| > \frac{1}{100}} \widehat{u}(Kt)J_0\left(4\pi m\sqrt{\frac{1}{q^2c_1^2} + \frac{1}{q^2c_2^2} \pm 2\frac{\cos(2\pi t)}{q^2c_1c_2}}\right)\mathrm{d}t\right|
\\\leq 
\int_{|t| > \frac{1}{100}} \left|\widehat u(Kt)\right| \mathrm dt
\ll_N K^{-N}.
\end{multline*}
Thus
\begin{multline}\label{eq:I-truncated}
\mathcal I_{m,K,q}^\pm(c_1,c_2)
\\=
\int_{|t|<\frac{1}{100}}\widehat{u}(Kt)J_0\left(4\pi m\sqrt{\frac{1}{q^2c_1^2} + \frac{1}{q^2c_2^2} \pm 2\frac{\cos(2\pi t)}{q^2c_1c_2}}\right)\mathrm{d}t
+
O_N\left(K^{-N}\right),
\end{multline}
and we are left with the task of bounding
\begin{equation*}  
\int_{|t|<\frac{1}{100}}\widehat{u}(Kt)J_0\left(4\pi m\sqrt{\frac{1}{q^2c_1^2} + \frac{1}{q^2c_2^2} \pm 2\frac{\cos(2\pi t)}{q^2c_1c_2}}\right)\mathrm{d}t.
\end{equation*}

We use the Taylor expansion:
$$
\cos(2\pi t) = 1 + \sum_{1\leq n < N} (-1)^n \frac{(2\pi t)^{2n}}{(2n)!}
+ O\left(t^{2N}\right).
$$ 
Plugging this into \eqref{eq:I-truncated}, we get 
\begin{multline}\label{eq:I-expression-j0-taylor}
\mathcal I_{m,K,q}^\pm(c_1,c_2)
=\\
\mathop{\int}_{|t|<\frac{1}{100}}\widehat{u}(Kt)J_0\left(\frac{4\pi m}{q}\sqrt{\left(\frac{1}{c_1} \pm \frac{1}{c_2}\right)^2 \pm \frac{2}{c_1c_2}\left(\sum_{1\leq n < N} (-1)^n \frac{(2\pi t)^{2n}}{(2n)!}
+ O\left(t^{2N}\right)\right)}\right)\mathrm{d}t
\\+
O_N\left(K^{-N}\right).
\end{multline}
We denote
$$
A = A^\pm(c_1,c_2,m,q) = \frac{(2\pi m)^2 (c_1\pm c_2)^2}{(c_1c_2)^2 q^2}
$$
and 
$$
B(t) = B^\pm(t;c_1,c_2,m,q)= \pm \frac{2(2\pi m)^2}{c_1c_2q^2}
\left(
\sum_{1\leq n \leq N} (-1)^n \frac{(2\pi t)^{2n}}{(2n)!}
+ O\left(t^{2N}\right)
\right).
$$
With this choice of \(A\) and \(B\), (\ref{eq:I-expression-j0-taylor}) reads
\begin{equation}\label{eq:I-expression-j0-AB}
\mathcal I_{m,K,q}^{\pm}(c_1,c_2)=\int_{|t|<\frac{1}{100}} \widehat u(Kt) J_0\left(2\sqrt{A+B(t)}\right)\mathrm dt
+
O_N\left(K^{-N}\right).
\end{equation}
We now apply \Cref{lem:bessel0_adition_theorem} to $J_0\left(2\sqrt{A+B(t)}\right)$. 
This shows that for any positive integer $N'$,
\begin{equation}\label{eq:j0-sqrt-ab}
J_0\left(2\sqrt{A+B(t)}\right)
=
\sum_{0\leq \ell < N'}
\frac{(-1)^\ell}{\ell!}
\left(\frac{B(t)}{\sqrt{A}}\right)^\ell J_\ell\left(2\sqrt{A}\right)
+ O\left(\left(\frac{|B(t)|}{\sqrt{A}}\right)^{N'}\right).
\end{equation}
Let $R = \frac{4\pi m}{q|c_1\pm c_2|}$. Then for $c_1 \neq c_2$ and $|t|<\frac{1}{100}$, one has
\begin{equation}\label{eq:B-over-sqrtA}
\frac{B(t)}{\sqrt{A}}
=
\pm R
\left(\sum_{1\leq n< N}\frac{(-1)^n (2\pi t)^{2n}}{(2n)!}+ O\left(t^{2N}\right)\right)
\ll 
Rt^2.
\end{equation}
Replacing this expression for $B(t)/\sqrt{A}$ in \eqref{eq:j0-sqrt-ab}, we get that
\begin{multline*}
J_0\left(2\sqrt{A + B(t)}\right) 
\\
=
\sum_{0\leq \ell<N'} \frac{(\mp R)^\ell}{\ell!} \left(\sum_{1\leq n<N}\frac{(-1)^n (2\pi t)^{2n}}{(2n)!}
+ O\left(t^{2N}\right) \right)^\ell
J_\ell\left(2\sqrt{A}\right)
+ 
O\left( R^{N'}t^{2N'}\right).
\end{multline*}
Multiplying out the $\ell$\textsuperscript{th} powers, and using the bound given in \eqref{eq:B-over-sqrtA} we get that
\begin{equation}\label{eq:J_0-taylor-expansion1}
J_0\left(2\sqrt{A + B(t)}\right) 
\\
=
\sum_{0\leq n< NN'} C_n t^{2n}
+ O\left(
\sum_{1\leq \ell < N'}
\ell \cdot 
\frac{ R^\ell t^{2(\ell - 1)}t^{2N}}{\ell !}
\right)
+O\left( R^{N'}t^{2N'}\right)
\end{equation}
for some coefficients $C_n$ depending on $m,q,c_1$ and $c_2$ and satisfying 
\begin{equation}\label{eq:Cn-bounds}
C_n \ll_{N,N'} 1 + R^{N'}.
\end{equation}
Furthermore, by considering the two cases $Rt^2 \leq 1$ and $Rt^2 > 1$, we get
$$
\sum_{1\leq \ell < N'}
\frac{R^\ell t^{2(\ell - 1)}t^{2N}}{(\ell-1) !}
\ll Rt^{2N} + R^{N'}t^{2(N'+N)}.
$$
With this, \eqref{eq:J_0-taylor-expansion1} becomes
\begin{multline*}
J_0\left(2\sqrt{A + B(t)}\right) 
=
\sum_{0\leq n<NN'} C_n t^{2n}
+ O\left(Rt^{2N}\right)
+ O\left(R^{N'}t^{2(N'+N)}\right)
+ O\left( R^{N'}t^{2N'}\right)
\\
=
\sum_{0\leq n<NN'} C_n t^n
+ O\left(Rt^{2N}\right)
+ O\left( R^{N'}t^{2N'}\right).
\end{multline*}
For the last equality we used the fact that the error term $O\left(R^{N'}t^{2(N'+N)}\right)$ is dominated by $O\left( R^{N'}t^{2N'}\right)$ since $|t|\leq \frac{1}{100}$.

Plugging this expression for $J_0\left(2\sqrt{A + B(t)}\right)$ back into \eqref{eq:I-expression-j0-AB}, we get
\begin{multline}\label{eq:I-expression-taylor1}
\mathcal I_{m,K,q}^\pm(c_1,c_2)
=
\sum_{0\leq n<NN'}
C_n
\int_{|t|<\frac{1}{100}} \widehat{u}(Kt) t^{2n}\mathrm dt
\\+
O\left(
R
\int_{|t|<\frac{1}{100}} |\widehat{u}(Kt) t^{2N}|\mathrm dt
\right)
+
O\left(R^{N'}
\int_{|t|<\frac{1}{100}} |\widehat{u}(Kt) t^{2N'}|\mathrm dt
\right)
+ O_N\left( K^{-N}\right).
\end{multline}
Using \eqref{eq:fourier-bound} we have that
$$
\int_{|t|<\frac{1}{100}} |\widehat{u}(Kt) t^{2N}|\mathrm dt
\ll_N  K^{-2N-1} \ll_N K^{-N-1}
$$
and
$$
\int_{|t|<\frac{1}{100}} |\widehat{u}(Kt) t^{2N'}|\mathrm dt
\ll_{N'} 
K^{-2N'-1}.
$$
With this \eqref{eq:I-expression-taylor1} becomes
\begin{multline}\label{eq:I-expression-taylor2}
\mathcal I_{m,K,q}^\pm(c_1,c_2)
=
\sum_{0\leq n<NN'}
C_n
\int_{|t|<\frac{1}{100}} \widehat{u}(Kt) t^{2n}\mathrm dt
\\+
O_N\left(
RK^{-N-1}
\right)
+
O_{N'}\left(
R^{N'}K^{-2N'-1}
\right)
+ O_N\left( K^{-N}\right).
\end{multline}
We next bound the sum over \(n\) above.
From \eqref{eq:fourier-derivative} we have that
$$
\int_{|t|<\frac{1}{100}} \widehat{u}(Kt) t^{2n}\mathrm dt
=
-\int_{|t|>\frac{1}{100}} \widehat{u}(Kt) t^{2n}\mathrm dt.
$$
Using \eqref{eq:u-decay}, which gives \(\widehat{u}(Kt)\ll_{N,N'} (K|t|)^{-100N'-N}\) for \(|t|\geq 1/100\), we bound this last integral by 
\begin{multline*}
\left|\int_{|t|>\frac{1}{100}} \widehat{u}(Kt) t^{2n}\mathrm dt\right|
\leq 
\int_{|t|>\frac{1}{100}} \left|\widehat{u}(Kt)\right| t^{2n}\mathrm dt
\\
\ll_{N,N'}
K^{-100N'-N}\int_{|t|>\frac{1}{100}} |t|^{-100N'-N+2n}\mathrm dt
\ll_{N,N',n} K^{-100N'-N}.
\end{multline*}
In combination with the bound \eqref{eq:Cn-bounds} for the coefficients $C_n$, we get
\begin{multline*}
\sum_{0\leq n<NN'}
C_n
\int_{|t|<\frac{1}{100}} \widehat{u}(Kt) t^{2n}\mathrm dt
\ll_{N,N'} \sum_{0\leq n<NN'}(1+R^{N'})K^{-100N'-N}\\
\ll_{N,N'} K^{-N} + R^{N'}K^{-100N'-N}.
\end{multline*}
With this, \eqref{eq:I-expression-taylor2} now becomes
\begin{multline}\label{eq:I-expression-bound}
\mathcal I_{m,K,q}^\pm(c_1,c_2)
\ll_{N,N'}
R^{N'}K^{-100N'-N}
+
RK^{-N-1}
+
R^{N'}K^{-2N'-1}
+
K^{-N}.
\end{multline}
Since $R\ll m \ll K^{100}$, the error term $RK^{-N-1}$ is bounded by $K^{99-N}$, and the error term $R^{N'}K^{-100N'-N}$ is bounded by $K^{-N}$.
And so \eqref{eq:I-expression-bound} reads
$$
\mathcal I_{m,K,q}^\pm(c_1,c_2)
\ll_{N,N'}
R^{N'}K^{-2N'-1}
+ K^{99 -N}
\ll_{N,N'}
\frac{1}{K}
\left(\frac{4 \pi m}{q|c_1 \pm c_2|K^2}\right)^{N'}
+
K^{99 -N}.
$$
\Cref{eq:I_bound} now follows after replacing $N$ by $N+99$.
\end{proof}

Equipped with these bounds for the integrals \(\mathcal I_{m,K,q}^\pm(c_1,c_2)\) and the Weil bound for the Kloosterman sums (\Cref{lem:twisted_kloosterman_bound}), we are now in a position to prove \Cref{prop:seond-moment-weight}.

\begin{proof}[Proof of \Cref{prop:seond-moment-weight}]

We return to the expression for $\sigma_K^2$ given in \Cref{lem:sigmaK-transformed}.
Let $\epsilon > 0$ and write 
\begin{equation}\label{eq:sigmaK-s1-s2-decomposition}
\sigma_K^2
=
(-1)^\delta 8\pi^2 (T_1+T_2)
+ O\left(e^{-K/2}\right), 
\end{equation}
where
\begin{multline*}
T_1=
\\
\sum_{\substack{1\leq c_1,c_2 \leq \frac{200m}{Kq}\\ |c_1 -c_2|\geq \frac{m}{K^{2-\epsilon}q}}}
\frac{S_\chi(m,m,qc_1)}{qc_1}
\frac{S_\chi(m,m,qc_2)}{qc_2}
\left( \mathcal I_{m,K,q}^-(c_1,c_2) - (-1)^\delta \mathcal I_{m,K,q}^+(c_1,c_2)\right),
\end{multline*}
and 
\begin{multline*}
T_2=
\\
\sum_{\substack{1\leq c_1,c_2 \leq \frac{200m}{Kq}\\ |c_1 -c_2|< \frac{m}{K^{2-\epsilon}q}}}
\frac{S_\chi(m,m,qc_1)}{qc_1}
\frac{S_\chi(m,m,qc_2)}{qc_2}
\left( \mathcal I_{m,K,q}^-(c_1,c_2) - (-1)^\delta \mathcal I_{m,K,q}^+(c_1,c_2)\right).
\end{multline*}

We first bound $T_1$. 
Note that since $c_1,c_2\geq 1$, the condition $|c_1 - c_2|\geq m/(K^{2-\epsilon}q)$ implies $|c_1+c_2|\geq m/(K^{2-\epsilon}q)$ also.
Therefore for all $c_1$ and $c_2$ in the range of summation of $T_1$, the bound \eqref{eq:I_bound} of \Cref{lem:Ibounds} gives, by choosing $N$ and $N'$ sufficiently large in terms of $A$ and $\epsilon$:
$$
\mathcal I_{m,K,q}^\pm(c_1,c_2)\ll_{A,\epsilon} K^{-A} \text{ for any } A\geq0.
$$
Bounding the Kloosterman sums trivially now shows $T_1 \ll_{A,\epsilon} \big(\frac{m}{Kq}\big)^2 K^{-A}$ for all $A\geq 0$.
Since we assume $m\leq K^{100}$, we conclude
\begin{equation}\label{eq:s1_bound}
T_1 \ll_{A,\epsilon} q^{-2} K^{-A} \; \text{ for any } A\geq 0.
\end{equation}

We next bound $T_2$.
Using the bound \eqref{eq:I-trivial-bound} of \Cref{lem:Ibounds} for $\mathcal I_{m,K,q}^\pm(c_1,c_2)$, we obtain 
\begin{equation*}
T_2
\ll 
\frac{1}{K}
\sum_{\substack{1\leq c_1,c_2 \leq \frac{200m}{Kq}\\ |c_1 -c_2|< \frac{m}{K^{2-\epsilon}q}}}
\frac{|S_\chi(m,m,qc_1)|}{qc_1}
\frac{|S_\chi(m,m,qc_2)|}{qc_2}.
\end{equation*}
Writing \(c_2=c_1+h\), and bounding the Kloosterman sums using \Cref{lem:twisted_kloosterman_bound} shows
\begin{multline}\label{eq:s2minus-estimate-weil-bound}
T_2
\ll
\frac{1}{K}
\sum_{0\leq h \leq \frac{m}{K^{2-\epsilon}q}}
\sum_{1\leq c\leq  \frac{200m}{Kq}}
\frac{\divis(qc)(m,qc)^{1/2}c^{1/2}}{qc}
\\ \times
\frac{\divis(q(c+h))(m,q(c+h))^{1/2}(c+h)^{1/2}}{q(c+h)}
q g(q')^2.
\end{multline}
We use the estimates $\divis(n) \ll_\epsilon n^\epsilon$ and $(m,qc) \leq (m,q)(m,c)$, which give
$$
T_2
\ll_\epsilon
\frac{(m,q) g(q')^2}{Kq^{1-2\epsilon}}
\sum_{0\leq h \leq \frac{m}{K^{2-\epsilon}q}}
\sum_{1\leq c\leq \frac{200m}{Kq}}
\frac{(m,c)^{1/2}}{c^{1/2-\epsilon}}
\cdot
\frac{(m,c+h)^{1/2}}{(c+h)^{1/2-\epsilon}}.
$$
Applying the Cauchy-Schwarz inequality to the sum over \(h\), it follows
\begin{multline}\label{eq:s2-cauchy-schwarz}
T_2 
\ll_\epsilon
\frac{(m,q) g(q')^2}{Kq^{1-2\epsilon}}
\sum_{0\leq h \leq \frac{m}{K^{2-\epsilon}q}}
\left( \sum_{1\leq c\leq \frac{200m}{Kq}} \frac{(m,c)}{c^{1-2\epsilon}}\right)^{1/2}
\left( \sum_{1\leq c\leq \frac{200m}{Kq}} \frac{(m,c+h)}{(c+h)^{1-2\epsilon}}\right)^{1/2}
\\
\ll_\epsilon
\frac{(m,q) g(q')^2}{Kq^{1-2\epsilon}} 
\cdot
\frac{m}{K^{2-\epsilon} q}
\sum_{c\leq \frac{400m}{Kq}}
\frac{(m,c)}{c^{1-2\epsilon}}.
\end{multline}
We estimate
$$
\sum_{ c \leq \frac{400 m}{Kq}}
\frac{(m,c)}{c^{1 - 2\epsilon}}
=
\sum_{d\mid m} \sum_{\substack{ c' \leq \frac{400m}{Kqd}\\ (c',m)=1}}\frac{d}{(c'd)^{1 - 2\epsilon}}
\ll 
\sum_{d\mid m} d^{2\epsilon} \left(\frac{m}{Kqd}\right)^{2\epsilon}
\ll_\epsilon \frac{m^{3\epsilon}}{(Kq)^{2\epsilon}}.
$$
Replacing this in \eqref{eq:s2-cauchy-schwarz} shows
\begin{equation}\label{eq:s2_bound}
T_2
\ll_\epsilon
\frac{m^{1+3\epsilon}(m,q) g(q')^2}{K^3 q^2}.
\end{equation}

Finally, we return to \eqref{eq:sigmaK-s1-s2-decomposition}. 
Applying the bound \eqref{eq:s1_bound} for \(T_1\) and the bound \eqref{eq:s2_bound} for \(T_2\), we conclude
$$
\sigma_K^2
\ll_\epsilon
\frac{m^{1+3\epsilon}(m,q) g(q')^2}{K^3 q^2}
+ K^{-1000}.
$$
Replacing \(\epsilon\) by \(\epsilon/3\) completes the proof of \Cref{prop:seond-moment-weight}.
\end{proof}
\section{Index Average: Proof of \texorpdfstring{\Cref{thm:m-avg}}{Theorem 1.2}}\label{sec:index_ave}

In this section we prove \Cref{prop:second-moment-index}, which, as shown in \cref{sec:deductions}, implies Theorem \ref{thm:m-avg}. 
Throughout this section, we let \(q\) and \(k\) be positive integers, \(M\geq1\) and \(\chi\) be a character modulo \(q\) satisfying \(\chi(-1)=(-1)^k\).
The second moment \(\sigma_M^2\) (defined in \eqref{eq:sigmaM-def}) depends upon these \(q,k, M\) and \(\chi\). 
All subsequent lemmas are uniform in these parameters, unless additional assumptions are stated explicitly.
Recall from the expression \eqref{eq:sigmaM-def} and the definition of \(\Delta_{k,q,\chi}\) given in \eqref{def:delta} that
\begin{align}\label{sigma0}
\sigma_M^2
=&(-1)^k \frac{4\pi^2}{M}
\sum_m 
u\left(\frac mM\right) 
\Delta_{k,q,\chi}(m,m)^2
\\=&
(-1)^k\frac {4\pi^2}{M}
\sum_{m}
u\left(\frac mM\right)
\sum_{c_1, c_2 \geq1} 
\frac{S_\chi(m,m;qc_1)}{qc_1}
\frac{S_\chi(m,m;qc_2)}{qc_2}
\nonumber\\
&\hspace{13em}\times
J_{k-1}\left(\frac{4\pi m}{qc_1}\right)
J_{k-1}\left(\frac{4\pi m}{qc_2}\right).
\nonumber
\end{align}

Our strategy to bound \(\sigma_M^2\) is essentially to open the Kloosterman sums and apply Poisson summation. 
This has the disadvantage of forgoing the Weil bound for the Kloosterman sums, but shortens the length of the \(m\)-sum, essentially saving a factor of \(M/k\). 
After Poisson, applying trivial bounds recovers the bound \(\sigma_M^2\ll Mf(q)  (Mkq)^\epsilon/(kq)^2\), which is essentially the same bound one obtains after bounding each term in the sum \(\sigma_M^2\) individually using the Weil bound. 
However, after Poisson we are able to find extra cancellation in the terms of the dual sum - bounding them non-trivially will give \Cref{prop:second-moment-index}. 

Our first step is to apply Poisson summation to the sum over \(m\) in (\ref{sigma0}). 

\begin{lemma}[Poisson summation] \label{lem:poisson_summation}
Suppose \(M\leq k^{1000}\). 
Then
\begin{equation*}
\sigma_M^2
=(-1)^k\frac {4\pi^2}{Mq^2} 
\sum_{c_1, c_2\leq \frac{200M}{kq}}
\frac1{c_1c_2}\sum_{n \in\mathbb{N}}
\mathcal N_\chi(n;qc_1,qc_2)
\mathcal I_{M,k}(n;qc_1,qc_2)
+O\left(e^{-k}\right).
\end{equation*}
Here, using the notation \(a^\dagger\equiv a+\overline a \bmod \ell\) we have for integers \(\ell_1\) and \(\ell_2\),
\begin{equation*}
\mathcal N_\chi(n;\ell_1,\ell_2)
\coloneqq 
\sum_{\substack{a_1 \bmod {\ell_1}\\a_2\bmod {\ell_2}\\(a_1,\ell_1)=(a_2,\ell_2)=1}}
\chi(a_1)\chi(a_2)
\mathbf{1}\{n\equiv a_1^\dagger \ell_2+a_2^\dagger \ell_1\bmod {\ell_1\ell_2}\}, 
\end{equation*}
and
\begin{equation*}
\mathcal I_{M,k}(n;\ell_1,\ell_2)
\coloneqq 
\int_{-\infty}^\infty 
u\left(\frac tM\right)
J_{k-1}\left(\frac{4\pi t}{\ell_1}\right)
J_{k-1}\left(\frac{4\pi t}{\ell_2}\right)
e\left(\frac{tn}{\ell_1\ell_2}\right)
\mathrm dt.
\end{equation*}
\end{lemma}

\begin{proof}

Our first step is to truncate the sums over \(c_1\) and \(c_2\) in (\ref{sigma0}). 
Using part \ref{besi} of \Cref{lem:bessel-bounds} and the assumption \(M\leq k^{1000}\), exactly as in \eqref{eq:sigmaK2-truncated} the contribution of \(c_1, c_2\geq 200M/(kq)\) to (\ref{sigma0}) is seen to be \(O(e^{-k})\). 
Thus from (\ref{sigma0}) we conclude 
\begin{multline*}
\sigma_M^2
=(-1)^k\frac {4\pi^2}{M}
\sum_{m}u\left(\frac mM\right) \\
\times \sum_{c_1, c_2 \leq \frac{200M}{kq}} 
\frac{S_\chi(m,m;qc_1)}{qc_1}
\frac{S_\chi(m,m;qc_2)}{qc_2}
J_{k-1}\left(\frac{4\pi m}{qc_1}\right)
J_{k-1}\left(\frac{4\pi m}{qc_2}\right)
+O\left(e^{-k}\right).
\end{multline*}
Interchanging the order of summation, we have
\begin{multline}\label{sigma1}
\sigma_M^2
=(-1)^k\frac {4\pi^2}{Mq^2}
\sum_{c_1, c_2 \leq \frac{200M}{kq}}
\frac{1}{c_1c_2}\\
\times\sum_{m} 
u\left(\frac mM\right)
S_\chi(m,m;qc_1)S_\chi(m,m;qc_2)
J_{k-1}\left(\frac{4\pi m}{qc_1}\right) 
J_{k-1}\left(\frac{4\pi m}{qc_2}\right)
+O\left(e^{-k}\right).
\end{multline}

It is now convenient to open the Kloosterman sums. 
Observe
\begin{multline*}
S_\chi(m,m;qc_1)S_\chi(m,m;qc_2)
=\sum_{\substack{a_1 \bmod {qc_1}\\a_2\bmod {qc_2}\\(a_1,qc_1)=(a_2,qc_2)=1}}
\chi(a_1)\chi(a_2)
e\left(m\left(\frac{a_1^\dagger}{qc_1}+\frac{a_2^\dagger}{qc_2}\right)\right)\\
=\sum_{b\bmod {q^2c_1c_2}}\mathcal N_\chi(b;qc_1, qc_2)
e\left(\frac{mb}{q^2c_1c_2}\right).
\end{multline*}
Replacing this in (\ref{sigma1}), we obtain
\begin{multline}\label{sigma2}
\sigma_M^2
=(-1)^k\frac {4\pi^2}{Mq^2}
\sum_{c_1, c_2 \leq \frac{200M}{kq}}
\frac{1}{c_1c_2}
\sum_{b\bmod {q^2c_1c_2}}\mathcal N_\chi(b;qc_1,qc_2)\\
\sum_m u\left(\frac mM\right)
J_{k-1}\left(\frac{4\pi m}{qc_1}\right)
J_{k-1}\left(\frac{4\pi m}{qc_2}\right)
e\left(\frac{mb}{q^2c_1c_2}\right)
+O\left(e^{-k}\right).
\end{multline}

We apply Poisson summation to the inner sum over \(m\) in (\ref{sigma2}). 
This shows
\begin{multline*}
\sum_{m\in\mathbb{Z}} 
u\left(\frac mM\right)
J_{k-1}\left(\frac{4\pi m}{qc_1}\right) 
J_{k-1}\left(\frac{4\pi m}{qc_2}\right)
e\left(\frac{mb}{q^2c_1c_2}\right)\\
=
\sum_{\widetilde m\in\mathbb{Z}}
\int_{-\infty}^\infty 
u\left(\frac tM\right)
J_{k-1}\left(\frac{4\pi t}{qc_1}\right)
J_{k-1}\left(\frac{4\pi t}{qc_2}\right)
e\left(\frac{tb}{q^2c_1c_2}-t\widetilde m\right)
\mathrm dt.
\end{multline*}
Consequently, \eqref{sigma2} implies
\begin{multline*}
\sigma_M^2
=
(-1)^k \frac {4\pi^2}{Mq^2} 
\sum_{c_1, c_2\leq \frac{200M}{kq}} 
\frac1{c_1c_2} 
\sum_{b\bmod{q^2c_1c_2}}\mathcal N_\chi(b;qc_1,qc_2)
\\ 
\times \sum_{\widetilde m}
\int_{-\infty}^\infty 
u\left(\frac tM\right)
J_{k-1}\left(\frac{4\pi t}{qc_1}\right)
J_{k-1}\left(\frac{4\pi t}{qc_2}\right)
e\left(\frac{t(b-\widetilde mq^2c_1c_2)}{q^2c_1c_2}\right)
\mathrm dt
+O\left(e^{-k}\right),
\end{multline*}
and so
\begin{multline*}
\sigma_M^2
=
(-1)^k \frac {4\pi^2}{Mq^2} 
\sum_{c_1, c_2\leq \frac{200M}{kq}} \frac1{c_1c_2}
\sum_{b\bmod{q^2c_1c_2}}
\mathcal N_\chi(b;qc_1,qc_2)
\\
\times
\sum_{\widetilde m}
\mathcal I_{M,k}(b-\widetilde mq^2c_1c_2;qc_1,qc_2)
+O\left(e^{-k}\right).
\end{multline*}
Note \(\mathcal N_\chi(b;qc_1,qc_2)\) depends only on the value of \(b\) modulo \(q^2c_1c_2\). 
In other words, \(\mathcal N_\chi(b;qc_1,qc_2)=\mathcal N_\chi(b-\widetilde mq^2c_1c_2;qc_1,qc_2)\). 
Writing the summation in terms of the new variable \(n=b-\widetilde m q^2c_1c_2\) shows
\begin{equation*}
\sigma_M^2
=
(-1)^k\frac {4\pi^2}{Mq^2}
\sum_{c_1, c_2\leq \frac{200M}{kq}}
\frac1{c_1c_2}
\sum_{n\in\mathbb{Z}}\mathcal 
N_\chi(n;qc_1,qc_2)
\mathcal I_{M,k}(n;qc_1,qc_2)
+O\left(e^{-k}\right),
\end{equation*}
as claimed.
\end{proof}

We next give straightforward bounds for the `arithmetic' and `analytic' parts of the variance, \(\mathcal N_\chi(n;qc_1,qc_2)\) and \(\mathcal I_{M,k}(n;qc_1,qc_2)\) respectively.

\begin{lemma}\label{lem:Nbound}

Let \(n\) be an integer and \(\ell_1\), \(\ell_2\) be positive integers. 
Then \(\mathcal N_\chi(n;\ell_1,\ell_2)=0\) if \((\ell_1,\ell_2)\nmid n\). 
Otherwise, one has the bound
\begin{equation*}
|\mathcal N_\chi(n;\ell_1, \ell_2)|
\leq 
(\ell_1,\ell_2) 
\divis(\ell_1)
\divis(\ell_2)
f(\ell_1)^{1/2}f(\ell_2)^{1/2}.
\end{equation*}
\end{lemma}

\begin{proof}
Trivially,
\begin{multline}\label{eq:solution_N_count} 
\left|\mathcal N_\chi(n;\ell_1,\ell_2)\right|
=
\left|
\sum_{\substack{a_1 \bmod {\ell_1}\\a_2\bmod {\ell_2}\\(a_1,\ell_1)=(a_2,\ell_2)=1}}
\chi(a_1)\chi(a_2)
\mathbf{1}\{n\equiv a_1^\dagger \ell_2+a_2^\dagger \ell_1\bmod {\ell_1\ell_2}\}
\right|
\\ \leq 
\sum_{\substack{a_1 \bmod {\ell_1}\\a_2\bmod {\ell_2}\\(a_1,\ell_1)=(a_2,\ell_2)=1}}\mathbf{1}\{n\equiv a_1^\dagger \ell_2+a_2^\dagger \ell_1\bmod {\ell_1\ell_2}\}.
\end{multline}
Our goal is to bound the count of solutions above.

Suppose \((\ell_1,\ell_2)=g\), and write \(\ell_1=g\ell_1'\), \(\ell_2=g\ell_2'\), so that \((\ell_1',\ell_2')=1\). 
If \(g\nmid n\) then the congruence \(n\equiv x\ell_2+y\ell_1\bmod{\ell_1\ell_2}\) has no solutions \((x,y)\). 
Hence \(\mathcal N_\chi(n;\ell_1,\ell_2)=0\). 

On the other hand, suppose \(n=gn'\). 
Then 
\begin{equation*}
n\equiv x\ell_2+y\ell_1\bmod{\ell_1\ell_2}
\iff 
n'\equiv x\ell_2'+y\ell_1'\bmod{g\ell_1'\ell_2'}.
\end{equation*}
Let \((x_0, y_0)\) be a solution to \(x_0\ell_2'+y_0\ell_1'=n'\), which exists since \((\ell_1',\ell_2')=1\). 
One verifies
\begin{multline*}
n'\equiv x\ell_2'+y\ell_1'\bmod{g\ell_1'\ell_2'}
\\ \iff 
x\equiv x_0+h\ell_1'\bmod{g\ell_1'} 
\text{ and } 
y\equiv y_0-h\ell_2'\bmod {g\ell_2'} 
\text{ for some } h\bmod g.
\end{multline*}
In other words, there are exactly \(g\) solutions \((x,y)\) to the congruence \(n\equiv x\ell_2+y\ell_1\bmod{\ell_1\ell_2}\), for which \(x\) and \(y\) are determined uniquely modulo \(\ell_1\) and \(\ell_2\) respectively. 
Call these solutions \((x_h, y_h)_{h\bmod g}\). 

The count of solutions given in (\ref{eq:solution_N_count}) is therefore
\begin{multline}\label{eq:solution_N_count2}
\sum_{\substack{a_1 \bmod {\ell_1}\\a_2\bmod {\ell_2}\\(a_1,\ell_1)=(a_2,\ell_2)=1}}
\mathbf{1}\{n\equiv a_1^\dagger \ell_2+a_2^\dagger \ell_1\bmod {\ell_1\ell_2}\}
\\
=\sum_{h \bmod g} 
\left(
\sum_{\substack{a_1 \bmod {\ell_1}\\(a_1,\ell_1)=1}}
\mathbf{1}\{a_1^\dagger \equiv x_h\bmod {\ell_1}\}
\right)
\left(
\sum_{\substack{a_2 \bmod {c_2}\\(a_2,\ell_2)=1}}
\mathbf{1}\{a_2^\dagger \equiv y_h\bmod {\ell_2}\}
\right).
\end{multline}
Observe \(a^\dagger \equiv x\bmod \ell\iff a+\overline a\equiv x\bmod \ell\implies a^2-xa+1\equiv 0\bmod \ell\). 
Denote 
$$
v(x;\ell)=\#\{a\bmod \ell: a^2-xa+1\equiv 0\bmod \ell\}.
$$
It now follows from (\ref{eq:solution_N_count}) and (\ref{eq:solution_N_count2}) that 
\begin{equation}\label{eq:solution_N_count3}
|\mathcal N_\chi(n;\ell_1,\ell_2)| \leq \sum_{h\bmod g} v(x_h;\ell_1)v(y_h;\ell_2).
\end{equation}

We claim that for any integer \(x\) and positive integer \(\ell\), \(v(x;\ell)\leq \divis(\ell)f(\ell)^{1/2}\).
The lemma follows immediately from this claim and (\ref{eq:solution_N_count3}). 
We now prove the claim. 
Firstly note that, by the Chinese Remainder Theorem, \(v(x;\ell)\) is a multiplicative function of \(\ell\). 
Consequently, it suffices to bound \(v(x; p^\alpha)\) for prime powers \(p^\alpha\).

First suppose \(p\) is odd. 
Then 
$$
v(x;p^\alpha)=\#\{a \bmod{p^\alpha}: a^2-xa+1\equiv 0\bmod {p^\alpha}\}.
$$
Completing the square, 
\begin{multline*}
a^2-xa+1\equiv 0\bmod {p^\alpha}
\iff 
4a^2-4xa+4\equiv 0\bmod {p^\alpha}
\\ \iff 
(2a-x)^2\equiv x^2-4\bmod {p^\alpha}.
\end{multline*}
So, changing variables to \(y=2a-x\), we have \(v(x;p^\alpha)=\#\{y\bmod{p^\alpha}: y^2\equiv x^2-4 \bmod p^\alpha\}\). 
This count of solutions can be computed explicitly, see for instance \cite[Lemma 9.6]{MR3099744}. 
For our purposes, the following cheap bound suffices. 
We can assume there exists a solution \(y_0^2\equiv x^2-4\bmod{p^\alpha}\), otherwise \(v(x;p^\alpha)=0\) and we are done. 
Then any \(y\) solving \(y^2\equiv x^2-4\bmod{p^\alpha}\) satisfies \(y^2\equiv y_0^2\bmod{p^\alpha}\iff p^\alpha\mid (y-y_0)(y+y_0)\). 
It follows
\[
\begin{cases} 
p^{(\alpha+1)/2}\mid (y-y_0)\text{ or } p^{(\alpha+1)/2}\mid (y+y_0)& \text{ if } \alpha \text{ is odd},\\ 
p^{\alpha/2}\mid (y-y_0)\text{ or } p^{\alpha/2}\mid (y+y_0)& \text{ if } \alpha \text{ is even}.
\end{cases}
\]
In particular, 
\[
\begin{cases} 
y\equiv \pm y_0\bmod{p^{(\alpha+1)/2}} & \text{ if } \alpha \text{ is odd},\\
y\equiv \pm y_0\bmod{p^{\alpha/2}} &\text{ if } \alpha \text{ is even}.
\end{cases}
\]
Thus, modulo \(p^\alpha\), \(y\) is restricted to \(2p^{(\alpha-1)/2}=2f(p^\alpha)^{1/2}\) residue classes if \(\alpha\) is odd, and \(2p^{\alpha/2}=2f(p^\alpha)^{1/2}\) residue classes if \(\alpha\) is even. 
Thus \(v(x;p^\alpha)\leq 2 f(p^\alpha)^{1/2}\) if \(p\) is an odd prime.

Finally, we bound \(v(x;2^\alpha)\). 
Completing the square, 
\begin{multline*}
a^2-xa+1\equiv0\bmod{2^\alpha}
\iff 
4a^2-4xa+4\equiv 0\bmod {2^{\alpha+2}}
\\
\iff (2a-x)^2\equiv x^2-4 \bmod{2^{\alpha+2}}.
\end{multline*}
So, noting \((2a-x)^2\) modulo \(2^{\alpha+2}\) depends only on the value of \(a\) modulo \(2^\alpha\), we have
\begin{multline}\label{eq:mod2-counting}
v(x;2^\alpha)
=
\#\SET{a\bmod {2^\alpha}: (2a-x)^2\equiv x^2-4\bmod {2^{\alpha+2}}}\\
= \frac 12
\#\SET{a\bmod {2^{\alpha+1}}: (2a-x)^2\equiv x^2-4\bmod {2^{\alpha+2}}}\\
\leq \frac 12
\#\SET{y\bmod{2^{\alpha+2}}:y^2\equiv x^2-4\bmod{2^{\alpha+2}}}.
\end{multline}
In the last step, we used that \(y=2a-x\) runs over exactly those residues \(y\) modulo \(2^{\alpha+2}\) satisfying \(y\equiv x\bmod 2\) as \(a\) runs over residues modulo \(2^{\alpha+1}\).
As in the case of odd primes, one bounds 
$$
\#\SET{y\bmod{2^{\alpha+2}}:y^2\equiv x^2-4\bmod{2^{\alpha+2}}}\leq 2f(2^{\alpha+2})^{1/2}=4 f(2^{\alpha})^{1/2}.
$$
It follows from this and \eqref{eq:mod2-counting} that \(v(x;2^\alpha)\leq 2f(2^\alpha)^{1/2}\). 

We have now shown \(v(x;p^\alpha)\leq 2f(p^\alpha)^{1/2}\) for all primes \(p\). We therefore conclude from the multiplicativity of \(f\) and \(v(\cdot;\ell)\) that 
$$
v(x; \ell)\leq 2^{\#\{p\mid \ell\}}f(\ell)^{1/2}\leq \divis(\ell)f(\ell)^{1/2},
$$
as claimed.
\end{proof}

\begin{lemma}\label{Ibound}
Let \(n\) be an integer, and let \(\ell_1, \ell_2\leq 200M/k\) be positive integers. 
Then for any integer \(A\geq0\), 
\begin{equation*}
\mathcal I_{M,k}(n;\ell_1,\ell_2)
\ll_A 
M \left(\frac{|n|}{\ell_1+\ell_2}\right)^{-A}.
\end{equation*}
\end{lemma}

\begin{proof}
This follows from a straightforward application of integration by parts.
To this end, we must differentiate the Bessel functions. 
For real \(\nu\), we have the identity (\cite[\textsection 3.2]{watson42})
\begin{equation*}
J_\nu'(z)=\frac12(J_{\nu-1}(z)-J_{\nu+1}(z)).
\end{equation*}
Using the bound \(|J_n(z)|\leq 1\) (valid for all integers \(n\), see \eqref{eq:bessel-intrep}), one has
\begin{equation*}
\frac {\mathrm d^A}{\mathrm dz^A} J_{k-1}(z)=\sum_{0\leq j\leq A} c(j;A) J_{k-1-(A+2j)}(z)\ll_A 1, 
\end{equation*}
where \(c(j;A)\) are some (computable) coefficients independent of \(k\). 
In particular,
\[
\frac{\mathrm d^A}{\mathrm dt^A} 
J_{k-1}\left(\frac {4\pi t}\ell\right)
\ll_A \ell^{-A}.
\]
Similarly, since \(u^{(A)}\ll_A 1\), one has
\[
\frac{\mathrm d^A}{\mathrm dt^A} 
u\left(\frac tM\right)
\ll_A M^{-A}.
\]
Combining the above calculations, we have
\begin{equation*}
\frac{\mathrm d^A}{\mathrm dt^A} 
\left\{
u\left(\frac tM\right) 
J_{k-1}\left(\frac{4\pi t}{\ell_1}\right) 
J_{k-1}\left(\frac{4\pi t}{\ell_2}\right)
\right\} 
\ll_A M^{-A}+\ell_1^{-A}+\ell_2^{-A}\ll_A \ell_1^{-A}+\ell_2^{-A}.
\end{equation*}
In the last step, we used the assumption \(\ell_1, \ell_2\leq 200M/k\ll M\). 
Finally, integrating \(\mathcal I_{M,k}(n;\ell_1,\ell_2)\) by parts \(A\) times, one obtains
\begin{align*}
\mathcal I_{M,k}(n;&\ell_1,\ell_2)
= 
\int_{-\infty}^\infty 
u\left(\frac tM\right)
J_{k-1}\left(\frac{4\pi t}{\ell_1}\right)
J_{k-1}\left(\frac{4\pi t}{\ell_2}\right)
e\left(\frac{tn}{\ell_1\ell_2}\right)
\mathrm dt\\
&=
\left(-\frac{2\pi i n}{\ell_1\ell_2}\right)^{-A}
\int_{-\infty}^\infty 
\frac{\mathrm d^A}{\mathrm dt^A}
\left\{
u\left(\frac tM\right) 
J_{k-1}\left(\frac{4\pi t}{\ell_1}\right) 
J_{k-1}\left(\frac{4\pi t}{\ell_2}\right)
\right\}
e\left(\frac{tn}{\ell_1\ell_2}\right)
\mathrm dt\\
&\ll_A 
\left(\frac{|n|}{\ell_1\ell_2}\right)^{-A}
M (\ell_1^{-A}+\ell_2^{-A})\\
&\ll_A 
M\left(\frac {|n|}{\ell_1}\right)^{-A}
+ M\left(\frac{|n|}{\ell_2}\right)^{-A}
\ll 
M\left(\frac{|n|}{\ell_1+\ell_2}\right)^{-A}.
\end{align*}
\end{proof}

We now return to \Cref{lem:poisson_summation}, and the bound for \(\sigma_M^2\) given there. 
To proceed, we bound the arithmetic terms \(\mathcal N_\chi(n;qc_1,qc_2)\) using \Cref{lem:Nbound}. 
To handle the integrals \(\mathcal I_{M,k}(n;qc_1,qc_2)\), we will use \Cref{Ibound} to show they are negligible for large enough \(n\).
For the remaining \(n\), we apply Cauchy-Schwarz to capture non-trivial
cancellation in the integrals.

\begin{lemma}\label{lem:sigma2-bound-after-cauchy}
For \(X\geq 1\), define 
\begin{equation}\label{s1def}
S_1(X;q)\coloneqq \sum_{c\leq X}\frac{f(qc)^{1/2}}{c^{1/4}}
\end{equation}
and 
\begin{equation}\label{s2def}
S_2(X;q)\coloneqq \sum_{c_1,c_2\leq X}\frac{(c_1,c_2)^{1/2}f(qc_1)^{1/2}f(qc_2)^{1/2}}{(c_1c_2)^{3/4}}.
\end{equation}
Suppose \(M\leq k^{1000}\), and let \(\epsilon>0\). 
Then
\begin{equation*}
\sigma_M^2\ll_\epsilon 
\frac{k^\epsilon\divis(q)^2}{M^{1/2}kq^{1/2}}
S_1\left(\frac{200M}{kq};q\right)^2+
\frac{k^{\epsilon}\divis(q)^2}{kq}
S_2\left(\frac{200M}{kq};q\right)
+k^{-1000}.
\end{equation*}
\end{lemma}

\begin{proof}
\Cref{lem:Nbound} gives that \(\mathcal N_\chi(n;qc_1,qc_2)=0\) whenever \((qc_1,qc_2)\nmid n\). 
From \Cref{lem:poisson_summation}, we thus obtain
\begin{multline*}
\sigma_M^2
=(-1)^k\frac {4\pi^2}{Mq^2} 
\sum_{c_1, c_2\leq \frac{200M}{kq}}
\frac1{c_1c_2}
\\
\times \sum_{{n \in\mathbb{N}}}
\mathcal N_\chi(q(c_1,c_2)n;qc_1,qc_2)
\mathcal I_{M,k}(q(c_1,c_2)n;qc_1,qc_2)
+O\left(e^{-k}\right).
\end{multline*}
Next, applying the bound of \Cref{lem:Nbound} and using that \(\divis(qc)\leq \divis(q)\divis(c)\), we have
\begin{multline}\label{sigma6}
|\sigma_M^2| \leq 
\frac {4\pi^2}{Mq^2} 
\sum_{c_1, c_2\leq \frac{200M}{kq}}
\frac{q(c_1,c_2)\divis(q)^2\divis(c_1)\divis(c_2)f(qc_1)^{1/2}f(qc_2)^{1/2}}{c_1c_2}
\\
\times \sum_n \left|\mathcal I_{M,k}(q(c_1,c_2)n;qc_1,qc_2)\right| +O\left(e^{-k}\right).
\end{multline}
We next use \Cref{Ibound} to truncate the inner sum over \(n\).
Indeed, since \(c_1,c_2\leq 200M/(kq)\), applying \Cref{Ibound} shows
\begin{multline*}
|n|\geq 
\frac{M}{k^{1-\epsilon}q(c_1,c_2)}
\\ \implies 
\mathcal I_{M,k}(q(c_1,c_2)n;qc_1,qc_2)
\ll_A 
M\left(\frac{q(c_1,c_2)|n|}{q(c_1+c_2)}\right)^{-A}
\ll_A 
\frac{M^3}{k^2q^2 |n|^2}\cdot k^{-(A-2)\epsilon},
\end{multline*}
for any integer \(A\geq 2\). 
Choosing \(A\) sufficiently large in terms of \(\epsilon\), and using the assumption \(M\leq k^{1000}\) and the trivial bound \(f(qc)\leq qc\), we conclude that the contribution to \eqref{sigma6} of those \(n\) with \(|n|\geq M/(k^{1-\epsilon}q(c_1,c_2))\) is \(O_\epsilon(k^{-1000})\), say. 
Thus 
\begin{multline}\label{sigma7}
|\sigma_M^2|
\leq 
\frac {4\pi^2\divis(q)^2}{Mq} 
\sum_{c_1, c_2\leq \frac{200M}{kq}}
\frac{(c_1,c_2)\divis(c_1)\divis(c_2)f(qc_1)^{1/2}f(qc_2)^{1/2}}{c_1c_2}
\\
\times \sum_{|n|\leq \frac{Mk^\epsilon}{kq(c_1,c_2)}} 
|\mathcal I_{M,k}(q(c_1,c_2)n;qc_1,qc_2)|
+O_\epsilon\left(k^{-1000}\right).
\end{multline}

The key idea of the proof is now to apply Cauchy-Schwarz to the inner sum over \(n\) in (\ref{sigma7}). 
Also using the bound \(\divis(c)\ll_\epsilon c^{\epsilon/4000}\ll_\epsilon M^{\epsilon/4000}\ll_\epsilon k^{\epsilon/4}\), valid for \(c\leq 200 M/(kq)\), we obtain 
\begin{multline}\label{sigma8}
\sigma_M^2
\ll_\epsilon 
\frac{k^\epsilon\divis(q)^2}{M^{1/2}k^{1/2}q^{3/2}}
\sum_{c_1, c_2\leq \frac{200M}{kq}}
\frac{(c_1,c_2)^{1/2}f(qc_1)^{1/2}
f(qc_2)^{1/2}}{c_1c_2}
\\
\times \left(\sum_{|n|\leq \frac{Mk^\epsilon}{kq(c_1,c_2)}} |\mathcal I_{M,k}(q(c_1,c_2)n;qc_1,qc_2)|^2\right)^{1/2}
+k^{-1000}.
\end{multline}

In light of (\ref{sigma8}), our goal is now to bound
\begin{multline}
\sum_{|n|\leq \frac{Mk^\epsilon}{kq(c_1,c_2)}} 
|\mathcal I_{M,k}(q(c_1,c_2)n;qc_1,qc_2)|^2
\\=
\iint 
u\left(\frac{t_1}{M}\right)
u\left(\frac{t_2}{M}\right)
J_{k-1}\left(\frac{4\pi t_1}{qc_1}\right)
J_{k-1}\left(\frac{4\pi t_1}{qc_2}\right)
J_{k-1}\left(\frac{4\pi t_2}{qc_1}\right)
J_{k-1}\left(\frac{4\pi t_2}{qc_2}\right) 
\\ \times 
\sum_{|n|\leq \frac{Mk^\epsilon}{kq(c_1,c_2)}}
e\left(\frac{(c_1,c_2)n(t_1-t_2)}{qc_1c_2}\right)
\mathrm dt_1\mathrm dt_2.
\end{multline}
Recall \(\mathrm{supp}(u)\subset [1/2,3]\), and \(u\ll 1\). 
Applying Cauchy-Schwarz twice to the double integral above therefore shows
\begin{align*}
\sum_{|n|\leq \frac{Mk^\epsilon}{kq(c_1,c_2)}}&
|\mathcal I_{M,k}(q(c_1,c_2)n;qc_1,qc_2)|^2
\\
\ll& 
\left(
\iint_{M/2}^{3M}
\left|J_{k-1}\left(\frac{4\pi t_1}{qc_1}\right)\right|^4
\left|\sum_{|n|\leq \frac{Mk^\epsilon}{kq(c_1,c_2)}} 
e\left(\frac{(c_1,c_2)n|t_1-t_2|}{qc_1c_2}\right)\right|
\mathrm dt_1\mathrm dt_2
\right)^{1/4}\\
\times & 
\left(
\iint_{M/2}^{3M}
\left|J_{k-1}\left(\frac{4\pi t_1}{qc_2}\right)\right|^4
\left|\sum_{|n|\leq \frac{Mk^\epsilon}{kq(c_1,c_2)}} 
e\left(\frac{(c_1,c_2)n|t_1-t_2|}{qc_1c_2}\right)\right|
\mathrm dt_1\mathrm dt_2
\right)^{1/4}\\
\times & 
\left(
\iint_{M/2}^{3M}
\left|J_{k-1}\left(\frac{4\pi t_2}{qc_1}\right)\right|^4
\left|\sum_{|n|\leq \frac{Mk^\epsilon}{kq(c_1,c_2)}}
e\left(\frac{(c_1,c_2)n|t_1-t_2|}{qc_1c_2}\right)\right|
\mathrm dt_1\mathrm dt_2
\right)^{1/4}\\
\times & 
\left(
\iint_{M/2}^{3M}
\left|J_{k-1}\left(\frac{4\pi t_2}{qc_2}\right)\right|^4
\left|\sum_{|n|\leq \frac{Mk^\epsilon}{kq(c_1,c_2)}} 
e\left(\frac{(c_1,c_2)n|t_1-t_2|}{qc_1c_2}\right)\right|
\mathrm dt_1\mathrm dt_2
\right)^{1/4}.
\end{align*}
By symmetry of \(t_1\) and \(t_2\), we have
\begin{align*}
\sum_{|n|\leq \frac{Mk^\epsilon}{kq(c_1,c_2)}}& 
|\mathcal I_{M,k}(q(c_1,c_2)n;qc_1,qc_2)|^2
\\
\ll & 
\left(
\iint_{M/2}^{3M}
\left|J_{k-1}\left(\frac{4\pi t_1}{qc_1}\right)\right|^4
\left|\sum_{|n|\leq \frac{Mk^\epsilon}{kq(c_1,c_2)}}
e\left(\frac{(c_1,c_2)n|t_1-t_2|}{qc_1c_2}\right)\right|
\mathrm dt_1\mathrm dt_2
\right)^{1/2}\\
\times & 
\left(
\iint_{M/2}^{3M}
\left|J_{k-1}\left(\frac{4\pi t_1}{qc_2}\right)\right|^4
\left|\sum_{|n|\leq \frac{Mk^\epsilon}{kq(c_1,c_2)}} 
e\left(\frac{(c_1,c_2)n|t_1-t_2|}{qc_1c_2}\right)\right|
\mathrm dt_1\mathrm dt_2
\right)^{1/2}.
\end{align*}
In both integrals above, we make the change of variables \((t,y)=(t_1,t_1-t_2)\). 
This gives
\begin{align}\label{eq:cauchy-iint}
 \sum_{|n|\leq \frac{Mk^\epsilon}{kq(c_1,c_2)}} & 
|\mathcal I_{M,k}(q(c_1,c_2)n;qc_1,qc_2)|^2
\\ \nonumber
\ll &
\left(
\int_{M/2}^{3M}
\left|J_{k-1}\left(\frac{4\pi t}{qc_1}\right)\right|^4
\int_{t-3M}^{t-M/2} 
\left|
\sum_{|n|\leq \frac{Mk^\epsilon}{kq(c_1,c_2)}} 
e\left(\frac{(c_1,c_2)n|y|}{qc_1c_2}\right)
\right|
\mathrm dy\mathrm dt
\right)^{1/2}
\\ \nonumber
\times & 
\left(
\int_{M/2}^{3M}
\left|J_{k-1}\left(\frac{4\pi t}{qc_2}\right)\right|^4
\int_{t-3M}^{t-M/2}
\left|
\sum_{|n|\leq \frac{Mk^\epsilon}{kq(c_1,c_2)}} 
e\left(\frac{(c_1,c_2)n|y|}{qc_1c_2}\right)
\right|
\mathrm dy\mathrm dt
\right)^{1/2}.
\end{align}
For a given \(M/2\leq t\leq 3M\), the inner integral is
\begin{multline}\label{eq:y_integral}
\int_{t-3M}^{t-M/2}
\left|
\sum_{|n|\leq \frac{Mk^\epsilon}{kq(c_1,c_2)}} 
e\left(\frac{(c_1,c_2)n|y|}{qc_1c_2}\right)
\right|
\mathrm dy
\leq 
\int_{-3M}^{3M}
\left|
\sum_{|n|\leq \frac{Mk^\epsilon}{kq(c_1,c_2)}} 
e\left(\frac{(c_1,c_2)n|y|}{qc_1c_2}\right)
\right|
\mathrm dy
\\=
\frac{qc_1c_2}{(c_1,c_2)}
\int_{-\frac{3M(c_1,c_2)}{qc_1c_2}}^{\frac{3M(c_1,c_2)}{qc_1c_2}}
\left|
\sum_{|n|\leq \frac{Mk^\epsilon}{kq(c_1,c_2)}} 
e(n|y|)
\right|
\mathrm dy.
\end{multline}
For any \(N\geq 1\) and any real \(\alpha\), one has the well-known bound
\[
\sum_{n\leq N}e(n\alpha)
\ll 
\min\{N, \|\alpha\|^{-1}\},
\]
where \(\|\alpha\|=\min_{n\in\mathbb{Z}}|\alpha-n|\) is the distance from \(\alpha\) to the nearest integer. 
Consequently, for any \(X\geq0\), we bound
\begin{multline*} 
\int_{-X}^X 
\left| \sum_{n\leq N} e(n|y|) \right|
\mathrm dy
\leq 
2\int_{0}^{\left\lfloor X\right\rfloor +3/2}
\left|\sum_{n\leq N} e(n|y|)\right|
\mathrm dy
\\ \ll 
\left(\left\lfloor X\right\rfloor+\frac32\right)
\int_0^{1/2}\min\{N, y^{-1}\}\mathrm dy
 \ll 
(X+1)
\left\{
\int_0^{1/N} N\mathrm dy+\int_{1/N}^{1/2}\frac {\mathrm dy}{y}
\right\}
\\ 
\ll (X+1)(\log (N)+1).
\end{multline*}
Using this bound in (\ref{eq:y_integral}) shows (for \(M/2\leq t\leq 3M\))
\begin{equation*}
\int_{t-3M}^{t-M/2}
\left|
\sum_{|n|\leq \frac{Mk^\epsilon}{kq(c_1,c_2)}} 
e\left(\frac{(c_1,c_2)n|y|}{qc_1c_2}\right)
\right|
\mathrm dy
\ll 
\left(M+\frac{qc_1c_2}{(c_1,c_2)}\right)\log (k).
\end{equation*}
Replacing this estimate in (\ref{eq:cauchy-iint}), it follows
\begin{multline} \label{eq:I-squared-cauchy-bound}
\sum_{|n|\leq \frac{Mk^\epsilon}{kq(c_1,c_2)}} 
|\mathcal I_{M,k}(q(c_1,c_2)n;qc_1,qc_2)|^2
 \ll
\left(M+\frac{qc_1c_2}{(c_1,c_2)}\right)
\log (k)
\\ \times
\left(
\int_{M/2}^{3M}
\left|J_{k-1}\left(\frac{4\pi t}{qc_1}\right)\right|^4
\mathrm dt
\right)^{1/2} 
\left(\int_{M/2}^{3M}
\left|J_{k-1}\left(\frac{4\pi t}{qc_2}\right)\right|^4
\mathrm dt
\right)^{1/2}.
\end{multline}

It remains to bound the integrals of Bessel functions appearing in (\ref{eq:I-squared-cauchy-bound}). 
To do so, we rewrite
\begin{multline}\label{eq:bessel-integral-4th-power}
\int_{M/2}^{3M}
\left|J_{k-1}\left(\frac{4\pi t}{qc}\right)\right|^4
\mathrm dt
=
\frac{qc}{4\pi}
\int_{2\pi M/(qc)}^{12\pi M/(qc)}
|J_{k-1}(t)|^4 \mathrm dt
\\\leq
\frac{qc}{4\pi}
\left\{
\int_0^{k-k^{1/3+\epsilon}}
+\int_{k-k^{1/3+\epsilon}}^{k+k^{1/3+\epsilon}}
+\int_{k+k^{1/3+\epsilon}}^{12\pi M/(qc)}
\right\}
|J_{k-1}(t)|^4
\mathrm dt.
\end{multline}
We use the bounds given in \Cref{lem:bessel-bounds} to bound each of the three integrals appearing in (\ref{eq:bessel-integral-4th-power}) (some of which may be empty, depending on the value of \(M/(qc)\)). 
Firstly, if \(t\leq k-k^{1/3+\epsilon}\), part \ref{besii} of \Cref{lem:bessel-bounds} gives \(J_{k-1}(t)\ll \exp(-k^{\epsilon})\). 
Secondly, if \(k-k^{1/3+\epsilon}\leq t\leq k+k^{1/3+\epsilon}\), part \ref{besiii} of \Cref{lem:bessel-bounds} gives \(J_{k-1}(t)\ll k^{-1/3}\).
Thirdly, if \(t\geq k+k^{1/3+\epsilon}\), part \ref{besiv} of \Cref{lem:bessel-bounds} gives \(J_{k-1}(t)\ll_\epsilon (t^2-k^2)^{-1/4}\). 
Replacing these bounds in (\ref{eq:bessel-integral-4th-power}), we obtain
\begin{multline*}
\int_{M/2}^{3M}
\left|J_{k-1}\left(\frac{4\pi t}{qc}\right)\right|^4
\mathrm dt
\\ 
\ll_\epsilon 
qc\left\{
\int_0^{k-k^{1/3+\epsilon}} \exp(-4k^{\epsilon})\mathrm dt+
\int_{k-k^{1/3+\epsilon}}^{k+k^{1/3+\epsilon}}k^{-4/3}\mathrm dt 
+\int_{k+k^{1/3+\epsilon}}^{12\pi M/(qc)} \frac{\mathrm dt}{(t+k)(t-k)}
\right\}
\\ 
\ll_\epsilon 
\frac{qc}{k^{1-\epsilon}}.
\end{multline*}
Applying this estimate in (\ref{eq:I-squared-cauchy-bound}), we conclude 
\begin{equation*} 
\sum_{|n|\leq \frac{Mk^\epsilon}{kq(c_1,c_2)}} 
|\mathcal I_{M,k}(q(c_1,c_2)n;qc_1,qc_2)|^2
\ll_\epsilon
\frac{q(c_1c_2)^{1/2}}{k^{1-2\epsilon}}
\left(M+\frac{qc_1c_2}{(c_1,c_2)}\right).
\end{equation*}
Replacing this estimate in (\ref{sigma8}) shows 
\begin{multline*}
\sigma_M^2
\ll_\epsilon
\frac{k^{2\epsilon}\divis(q)^2}{kq}
\sum_{c_1,c_2\leq \frac{200M}{kq}}
\frac{(c_1,c_2)^{1/2}f(qc_1)^{1/2}f(qc_2)^{1/2}}{(c_1c_2)^{3/4}}
\\ +
\frac{k^{2\epsilon}\divis(q)^2}{M^{1/2}kq^{1/2}}
\sum_{c_1,c_2\leq \frac{200M}{kq}}
\frac{f(qc_1)^{1/2}f(qc_2)^{1/2}}{(c_1c_2)^{1/4}}
+k^{-1000}.
\end{multline*}
Replacing \(\epsilon\) by \(\epsilon/2\) now gives the lemma. 
\end{proof}

We deduce \Cref{prop:second-moment-index} from \Cref{lem:sigma2-bound-after-cauchy} and the following bounds for \(S_1\) and \(S_2\).

\begin{lemma}\label{lem:S-sums-bounds}
Let \(q\) be a positive integer and \(X\geq 0\). 
Then
\begin{equation}\label{eq:s1-bound}
S_1(X;q)\ll f(q)^{1/2}\divis(q)X^{3/4}\log(X),
\end{equation}
and 
\begin{equation}\label{eq:s2-bound}
S_2(X;q)\ll q^{1/2}f(q)^{1/2}\divis(q)^{2}X^{1/2}\log^{18}(X).
\end{equation} 
\end{lemma}

\begin{proof}[Proof of \Cref{prop:second-moment-index}, assuming \Cref{lem:S-sums-bounds}]
Lemmas \ref{lem:sigma2-bound-after-cauchy} and \ref{lem:S-sums-bounds} show
\begin{align}\label{eq:sigmaM-bound}
\sigma_M^2
\ll_\epsilon & 
\frac{k^\epsilon\divis(q)^2}{M^{1/2}kq^{1/2}}
S_1\left(\frac{200M}{kq};q\right)^2
+\frac{k^{\epsilon}\divis(q)^2}{kq}
S_2\left(\frac{200M}{kq};q\right)
+k^{-1000}
\\ \nonumber
\ll_\epsilon &
\frac{k^\epsilon\divis(q)^2}{M^{1/2}kq^{1/2}}
\left(\frac{200M}{kq}\right)^{3/2+\epsilon}f(q)\divis(q)^2
\\ \nonumber
& \quad\quad + 
\frac{k^{\epsilon}\divis(q)^2}{kq}
\left(\frac{200M}{kq}\right)^{1/2+\epsilon}
q^{1/2}f(q)^{1/2}\divis(q)^2
+k^{-1000}.
\end{align}
Using the well-known fact \(\divis(q)\ll_\epsilon q^{\epsilon/4}\), the result now follows at once from \eqref{eq:sigmaM-bound}.
\end{proof}

\begin{proof}[Proof of \Cref{lem:S-sums-bounds}]

Throughout, we denote by $\alpha(n)$ the product of all primes 
dividing $n$ with odd multiplicity, so that $\alpha(n) = n/f(n)$.

The following identity is the cornerstone of this proof. 
For any integers $u$ and $v$, we have
\begin{equation}\label{eq:f(qc)-splitting}
f(uv)=f(u)f\left(\frac{v}{(v, \alpha(u))}\right)(v,\alpha(u))^2.
\end{equation}
Since $f$ is multiplicative, \eqref{eq:f(qc)-splitting} is verified by checking the cases in which $u$ and $v$ are powers of a given prime.

For $0<\sigma<1$ and any integer \(a\), we denote 
$$
T_\sigma(X;a) = \sum_{c\leq X}\frac{f(ca)^{1/2}}{c^\sigma}.
$$
We begin by proving a bound for $T_\sigma(X;a)$, which will be used in bounding both $S_1(X;q)$ and $S_2(X;q)$.
Enumerating over all possible values of $(c,\alpha(a))$ and using \eqref{eq:f(qc)-splitting} we get the bound
\begin{multline}\label{eq:T_sigma-bound1}
T_\sigma(X;a)
=
\sum_{d\mid \alpha(a)}\sum_{\substack{c\leq X\\ (c,\alpha(a))=d}} 
\frac{f(a)^{1/2}f(c/d)^{1/2} d}{c^{\sigma}}
\leq 
f(a)^{1/2}
\sum_{d\mid \alpha(a)}d
\sum_{\substack{c\leq X\\ d\mid c}} 
\frac{f(c/d)^{1/2}}{c^{\sigma}}
\\=
f(a)^{1/2}
\sum_{d\mid \alpha(a)}d^{1 - \sigma}
\sum_{c'\leq X/d} \frac{f(c')^{1/2}}{(c')^{\sigma}}.
\end{multline}
To bound the inner sum, set $c' = s^2r$ with $r$ square-free, so that $f(c') = s^2$.
This shows 
\begin{multline*}
\sum_{c'\leq X/d}\frac{f(c')^{1/2}}{(c')^\sigma}
\leq 
\sum_{s^2 \leq X/d}\frac{s}{s^{2\sigma}}\sum_{r \leq X/(ds^2)} \frac{1}{r^\sigma}
\ll
\sum_{s^2 \leq X/d}s^{1 - 2\sigma}\left(\frac{X}{ds^2}\right)^{1 - \sigma}
\\=
\frac{X^{1 - \sigma}}{d^{1-\sigma} }\sum_{s^2 \leq X/d }\frac{1}{s}
\ll
\frac{X^{1 - \sigma}\log(X)}{d^{1-\sigma}}.
\end{multline*}
Replacing this in \eqref{eq:T_sigma-bound1} gives
\begin{equation}\label{eq:T_sigma-bound-final}
T_\sigma(X;a)
\ll 
X^{1-\sigma} \log(X) f(a)^{1/2} \sum_{d\mid \alpha(a)}1
\leq
X^{1-\sigma} \log(X) f(a)^{1/2} \divis(a).
\end{equation}

This bound for $T_\sigma(X;a)$ provides the required bound for $S_1(X;q)$, since by definition
$$
S_1(X;q) = 
\sum_{c\leq X} \frac{f(qc)^{1/2}}{c^{1/4}}
=
T_{1/4}(X;q),
$$
so \eqref{eq:T_sigma-bound-final} gives the claimed bound
$$
S_1(X;q) \ll X^{3/4}\log(X) f(q)^{1/2}\divis(q).
$$

We next bound $S_2(X;q)$.
We remind the reader of the definition
\begin{equation*}
S_2(X;q) = 
\sum_{c_1,c_2\leq X}
\frac{f(qc_1)^{1/2}f(qc_2)^{1/2}(c_1,c_2)^{1/2}}{(c_1c_2)^{3/4}}.
\end{equation*}
We enumerate over all possible values of $(c_1,\alpha(q))$ and $(c_2,\alpha(q))$.
By \eqref{eq:f(qc)-splitting} we get
\begin{equation*}
S_2(X;q)=
f(q)
\sum_{d_1,d_2\mid \alpha(q)} d_1d_2
\sum_{\substack{c_1,c_2\leq X\\ (c_1, \alpha(q))=d_1 \\(c_2,\alpha(q))=d_2}}
\frac{f(c_1/d_1)^{1/2}f(c_2/d_2)^{1/2}(c_1,c_2)^{1/2}}{(c_1c_2)^{3/4}}.
\end{equation*}

Write $c_1=d_1c_1'$ and $c_2=d_2c_2'$. 
Note that if $p\mid d_1$ and $p\mid c_2'$ then $p\mid d_2$ as well. 
Since $d_1$ and $d_2$ are square-free, it follows that $(c_1,c_2)=(d_1,d_2)(c_1',c_2')$. 
Thus 
\begin{equation}\label{eq:s2-after-splitting}
S_2(X;q)
\leq 
f(q)
\sum_{d_1,d_2\mid \alpha(q)} (d_1d_2)^{1/4} (d_1,d_2)^{1/2}
\sum_{\substack{ c_1'\leq X/d_1\\ c_2'\leq X/d_2}} 
\frac{f(c_1')^{1/2}f(c_2')^{1/2}(c_1',c_2')^{1/2}}{(c_1'c_2')^{3/4}}.
\end{equation}
To handle the innermost sum over $c_1'$ and $c_2'$, set $g=(c_1', c_2')$ and write $c_1'=gc_1''$ and $c_2'=gc_2''$. 
We then have
\begin{align*}
\sum_{\substack{c_1'\leq X/d_1\\ c_2'\leq X/d_2}} &
\frac{f(c_1')^{1/2}f(c_2')^{1/2}(c_1',c_2')^{1/2}}{(c_1'c_2')^{3/4}}
\\ 
&\leq 
\sum_{g\leq X} \frac 1g
\left(\sum_{c_1''\leq X/(gd_1)}\frac{f(c_1''g)^{1/2}}{(c_1'')^{3/4}}\right)
\left(\sum_{c_2''\leq X/(gd_2)}\frac{f(c_2''g)^{1/2}}{(c_2'')^{3/4}}\right)
\\ 
& =
\sum_{g\leq X} \frac{1}{g}
T_{3/4}\left(\frac{X}{gd_1};g\right)
T_{3/4}\left(\frac{X}{gd_2};g\right).
\end{align*}
Using \eqref{eq:T_sigma-bound-final} we then get
\begin{multline*}
\sum_{\substack{c_1'\leq X/d_1\\ c_2'\leq X/d_2}} 
\frac{f(c_1')^{1/2}f(c_2')^{1/2}(c_1',c_2')^{1/2}}{(c_1'c_2')^{3/4}}
\\ \ll 
\sum_{g\leq X} \frac{1}{g} 
\left(\frac{X}{gd_1}\right)^{1/4}
\left(\frac{X}{gd_2}\right)^{1/4}
\log^2(X) f(g) \divis^2(g)
=
\frac{X^{1/2}\log^2(X)}{(d_1d_2)^{1/4}}
\sum_{g\leq X} \frac{f(g) \divis^2(g)}{g^{3/2}}.
\end{multline*}
Replacing this estimate in \eqref{eq:s2-after-splitting}, we obtain
\begin{equation}\label{eq:s2-nearly-there}
S_2(X;q)
\ll
X^{1/2}\log^2(X) f(q)
\sum_{d_1,d_2\mid \alpha(q)} (d_1,d_2)^{1/2}\sum_{g\leq X}\frac{f(g)\divis(g)^2}{g^{3/2}}.
\end{equation}
To bound the inner sum over $g$, write $g=s^2r$ with $r$ square-free, so that $f(g)=s^2$. 
Using the inequality $\divis(uv)\leq \divis(u)\divis(v)$, we see that
\begin{equation*}
\sum_{g\leq X} \frac{ f(g)\divis(g)^2}{g^{3/2}}
\leq 
\sum_{s^2\leq X} \frac{\divis(s)^4}{s}
\sum_{\substack{r\leq X/s^2\\ r\text{ square-free}}} \frac{\divis(r)^2}{r^{3/2}}.
\end{equation*}
The sum over $r$ converges and is bounded by some absolute constant.
As for the $s$ sum, we have the bound 
$$
\sum_{s\leq \sqrt{X}} \frac{\divis^4(s)}{s} \leq
\prod_{p\leq \sqrt X}
\left(1 + \frac{2^4}{p} + \frac{3^4}{p^2} + \frac{4^4}{p^3}+...\right) \ll
\log^{16}(X).
$$
Thus \eqref{eq:s2-nearly-there} shows
\begin{multline*}
S_2(X;q)
\ll 
X^{1/2}\log^{18}(X) f(q) \sum_{d_1, d_2\mid \alpha(q)} (d_1,d_2)^{1/2}
\\ \leq 
X^{1/2}\log^{18}(X) f(q)\alpha(q)^{1/2}\divis(q)^2=X^{1/2}\log^{18}(X) f(q)^{1/2}q^{1/2}\divis(q)^2,
\end{multline*}
as claimed. 
In the last step, we used the equality \(f(q)\alpha(q)=q\), which follows from the definitions of $f$ and $\alpha$. 
\end{proof}

\bibliographystyle{plain}
\bibliography{references}
\end{document}